\documentclass[11pt]{article}
\usepackage[a4paper, total={6in, 9in}]{geometry}
\usepackage{amsmath} 
\usepackage{mathtools} %
\usepackage{amssymb} 
\usepackage{graphicx}
\usepackage{enumitem}
\usepackage[font=small,labelfont=bf]{caption}
\usepackage{subcaption}
\usepackage{amsthm}
\usepackage{siunitx}
\usepackage{comment}
\usepackage{mleftright}
\usepackage{authblk}
\newtheoremstyle{break}
  {\topsep}{\topsep}%
  {\itshape}{}%
  {\bfseries}{}%
  {\newline}{}%
\theoremstyle{break}
\newtheorem{theorem}{Theorem}[section]
\newtheorem{corollary}[theorem]{Corollary}
\newtheorem{proposition}[theorem]{Proposition}
\newtheorem{lemma}[theorem]{Lemma}

\usepackage{hyperref}
\hypersetup{colorlinks = true, linkcolor = blue, filecolor = magenta, urlcolor = cyan, citecolor = blue, pdfpagemode = FullScreen}

\usepackage[
    backend=biber,        
    style=numeric,        
    sorting=nyt,          
    maxnames=9,           
    giveninits=true,      
    doi=true,             
    url=false,            
    isbn=false,           
    sortcites=true
]{biblatex}

\DeclareNameAlias{sortname}{family-given} 
\DeclareNameAlias{default}{family-given}

\DeclareFieldFormat{doi}{%
  \ifhyperref
    {\href{https://doi.org/#1}{DOI: \nolinkurl{#1}}}
    {DOI: \nolinkurl{#1}}}

\usepackage{orcidlink}

\DeclareMathOperator{\supp}{supp}
\DeclareMathOperator{\vspan}{span}
\DeclareMathOperator{\dist}{dist}

\DeclareMathOperator\Arg{Arg}

\title{Polynomial curvelets on higher-dimensional spheres}

\author{Frederic Schoppert \orcidlink{0000-0002-8682-3723}\,\thanks{\href{mailto:f.schoppert@uni-luebeck.de}{f.schoppert@uni-luebeck.de}}}
\affil{Institute of Mathematics, University of L{\"u}beck,
Ratzeburger Allee 160, 23562, L{\"u}beck, Germany}

\date{\vspace{-5ex}}

\begin{document}

\maketitle
\begin{abstract}
In this article, we introduce and investigate polynomial curvelets on spheres, which form a class of Parseval frames for $L^2(\mathbb{S}^{d-1})$, $d \geq 3$. The proposed construction offers a directionally sensitive multiscale decomposition and provides a sparse representation of spherical data. As a main result, we derive a sharp pointwise localization bound which shows that the frame elements decay rapidly away from their center of mass, making them a powerful tool for position-based analyses. In contrast to previous constructions, polynomial curvelets are not limited in their directional resolution. Consequently, the frames established in this article are particularly powerful when it comes to the analysis of localized anisotropic features, such as edges. To illustrate this point, we show that, given a suitable test signal that exhibits (higher-order) discontinuities at the boundary $\partial A$ of a spherical cap $A\subset \mathbb{S}^{d-1}$, the corresponding curvelet coefficients peak precisely when the analysis function matches some segment of the boundary $\partial A$, both in terms of position and orientation. Otherwise, the coefficients decay rapidly.
\end{abstract}



\section{Introduction}
In a recent article \cite{bib77}, we introduced a general class of polynomial frames for $L^2(\mathbb{S}^{d-1})$, $d\geq 3$. This class includes the well-established zonal (isotropic) systems, sometimes called spherical needlets (see e.g.\ \cite{bib11, bib3, bib15, bib16, bib29, bib30, bib14, bib46, bib48, bib49, bib25}), but also paves the way for a variety of new anisotropic (non-zonal) designs. Based on this newly established framework, we introduce a family of highly directional and localized Parseval frames for $L^2(\mathbb{S}^{d-1})$, $d \geq 3$. These frames extend the two-dimensional construction of second-generation curvelets presented in \cite{bib1}, and they provide a tool for position-frequency analyses of spherical data, offering additional directional resolution at each scale and position. Unlike previous constructions, we show that the polynomial curvelets defined in this article are not limited in their directional sensitivity and they can be used to precisely identify localized features, such as edges, not only in terms of their position but also in terms of their orientation. 

On the $2$-sphere, localized directional polynomial frames were previously established in the form of so-called directional wavelets \cite{bib9, bib21, bib10} and second-generation curvelets \cite{bib1}. Due to their excellent spatial localization and directional sensitivity, both systems have found a variety of applications in the analysis of spherical random fields (e.g.\ \cite{bib9, bib70, bib63, bib65, bib67, bib69}), often related to problems in cosmology. Further areas of relevance include sparse reconstruction (e.g.\ \cite{bib66}), image segmentation (e.g.\ \cite{bib68}), as well as the detection of (higher-order) discontinuities in signals \cite{bib26, bib31}. 

The general framework established in \cite{bib77} includes the aforementioned directional wave\-lets and second-generation curvelets as special cases and, moreover, allows us to extend these systems to higher-dimensional spheres. For directional wavelets, an extension to $\mathbb{S}^{d-1}$, $d \geq 3$, was studied in \cite{bib37}, resulting in a class of well-localized Parseval frames for $L^2(\mathbb{S}^{d-1})$, with adjustable (but limited) directional sensitivity. Here, the corresponding frames are generated by an initial sequence of polynomials $\Psi^j  \in \Pi_{2^j}(\mathbb{S}^{d-1})$, $j\in \mathbb{N}_0$, that are highly concentrated at the North Pole $e^d$. Specifically, it was shown that
\begin{equation}\label{localization bound isotr. introd}
    \lvert \Psi^j(\eta)\rvert \leq \frac{c_{q}2^{j(d-1)}}{(1+2^j \dist(\eta, e^d))^{q}}, \qquad \eta \in \mathbb{S}^{d-1},
\end{equation}
for each $q\geq1$ \cite[Theorem~4.1]{bib37}. For corresponding zonal frame functions, often called spherical needlets, this kind of localization bound is well-known (see \cite{bib3, bib15, bib16, bib81}). Moreover, in the special case $d=3$, the estimate \eqref{localization bound isotr. introd} was first proven in \cite{bib31} for directional wavelets on $\mathbb{S}^2$, and prior to that a slightly weaker version was shown in \cite{bib9}.

In this article, we extend the two-dimensional second-generation curvelets introduced in \cite{bib1} to higher-dimensional spheres, by following the general construction in \cite{bib77}. Specifically, for $d \geq 3$, we define an initial sequence $(\Psi_{\scriptscriptstyle \mathrm{C}}^j)_{j=0}^\infty$ of polynomial curvelets satisfying $\Psi_{\scriptscriptstyle \mathrm{C}}^0 \equiv 1$ and $\Psi_{\scriptscriptstyle \mathrm{C}}^j \in \Pi_{2^j}(\mathbb{S}^{d-1})$ for each $j \in \mathbb{N}$. By using suitable rotation matrices $g_{j, r} \in SO(d)$ with corresponding weights $\mu_{j, r}>0$, which arise from positive quadrature rules for spherical polynomials (see the discussion in \autoref{sec3}), we obtain tight (Parseval) frames for $L^2(\mathbb{S}^{d-1})$ that are of the form
\begin{equation*}
   \mathcal{X}[(\Psi_{\scriptscriptstyle \mathrm{C}}^j)_{j=0}^\infty] =\{ \sqrt{\mu_{j, r}}\,\Psi_{\scriptscriptstyle \mathrm{C}}^j(g_{j, r}^{-1}  \cdot ), \; r=1, \dots , r_j, \; j \in \mathbb{N}_0 \}.
\end{equation*}
Thus, each $f \in L^2(\mathbb{S}^{d-1})$ can be expanded into the polynomial curvelets $\Psi_{\scriptscriptstyle \mathrm{C}}^j(g_{j, r}^{-1}  \cdot)$ via 
\begin{equation*}
    f = \sum_{j=0}^\infty \sum_{r=1}^{r_j} \mu_{j, r} \langle f, \Psi_{\scriptscriptstyle \mathrm{C}}^j(g_{j, r}^{-1}  \cdot )\rangle_{\mathbb{S}^{d-1}} \Psi_{\scriptscriptstyle \mathrm{C}}^j(g_{j, r}^{-1}  \cdot ).
\end{equation*}
Indeed, we show that this frame expansion converges fast in $L^p(\mathbb{S}^{d-1})$, $1 \leq p<\infty$, and $C(\mathbb{S}^{d-1})$, provided that $f$ belongs to the corresponding Banach space.

As a main result of the investigations presented in this article, we prove that polynomial curvelets satisfy the direction-dependent localization bound
\begin{equation}\label{localization bound intr}
    \lvert \Psi_{\scriptscriptstyle \mathrm{C}}^j(x)\rvert \leq \frac{c_{q} 2^{j(3d-2)/4}\lvert x_d+\mathrm{i}x_{d-1}\rvert^{2^{j-1}} }{(1+2^j \lvert \Arg(x_d+\mathrm{i}x_{d-1}) \rvert)^q}, \qquad x=(x_1, \dots, x_d) \in \mathbb{S}^{d-1},
\end{equation}
for each $q \in \mathbb{N}$. Here, for $z \in \mathbb{C}\setminus \{0\}$, $\Arg(z)$ denotes the unique element in $(-\pi, \pi]$ such that $z = \lvert z \rvert \mathrm{exp}(\mathrm{i}\Arg(z))$. The upper bound in \eqref{localization bound intr} implies, in particular, that the functions $\Psi_{\scriptscriptstyle \mathrm{C}}^j$ are highly localized at the North Pole $e^d$. For two-dimensional curvelets \cite{bib1}, i.e., in the case $d=3$, this estimate was previously established in \cite[Proposition~2.4]{bib26}.

In contrast to other polynomial frames for $L^2(\mathbb{S}^{d-1})$, the polynomial curvelets constructed in this article are, just like in the two-dimensional setting \cite{bib1, bib26}, not limited in their directional sensitivity. To reinforce this point, we derive an explicit formula for the auto-correlation functions
\begin{equation*}
    SO(d-1) \rightarrow \mathbb{C}, \qquad h \mapsto\langle \Psi_{\scriptscriptstyle \mathrm{C}}^j, \Psi_{\scriptscriptstyle \mathrm{C}}^j(h^{-1} \cdot) \rangle_{\mathbb{S}^{d-1}}, \qquad j \in \mathbb{N}_0,
\end{equation*}
which we observe to be highly concentrated at the identity matrix for large scales $j$. Here, we identify $SO(d-1)$ with the subgroup $\{h \in SO(d) : he^d = e^d\}$ of $SO(d)$. The price paid for the extra directional resolution is an increased computational cost. Specifically, the number of rotations needed at each scale $j$ grows like $2^{j(2d-3)}$, whereas for previously considered polynomial frames it typically scales like $2^{j(d-1)}$.

To conclude this article, we demonstrate the effectiveness of polynomial curvelets when it comes to the detection of edges and higher-order singularities. Here, we consider a class of test signals $f$, which are smooth except for a (higher-order) discontinuity along the boundary $\partial A$ of some spherical cap $A \subset \mathbb{S}^{d-1}$. In the simplest case, $f$ is just the characteristic function of $A$, i.e., $f=\mathbf{1}_A$. We then prove that the frame coefficients $\langle f, \Psi_{\scriptscriptstyle \mathrm{C}}^j(g^{-1} \cdot ) \rangle_{\mathbb{S}^{d-1}}$, $g \in SO(d)$, peak when the rotated analysis function $\Psi_{\scriptscriptstyle \mathrm{C}}^j(g^{-1} \cdot )$ matches some segment of the boundary $\partial A$ in terms of position and orientation. If $\Psi_{\scriptscriptstyle \mathrm{C}}^j(g^{-1} \cdot )$ is located away from $\partial A$, or if the orientation of $\Psi_{\scriptscriptstyle \mathrm{C}}^j(g^{-1} \cdot )$ does not match the orientation of the closest edge segment, the inner products decay rapidly. The latter also underlines the (expected) sparsity of the curvelet representation. We note that in the two-dimensional context, i.e., for $d=3$, the results presented have previously been established in \cite{bib26}.

The remainder of this article is structured as follows. In \autoref{sec2}, we briefly review some fundamentals of harmonic analysis on $\mathbb{S}^{d-1}$, which will be essential for our investigations. The construction of Parseval frames consisting of polynomial curvelets is carried out in \autoref{sec3}. In \autoref{sec4}, we prove that the frame elements satisfy the direction-dependent localization bound in \eqref{localization bound intr}. In \autoref{sec5}, we derive explicit expressions for the auto-correlation functions associated with polynomial curvelets, which point to a particularly strong directional sensitivity. Finally, in \autoref{sec6}, we demonstrate the strength of polynomial curvelet frames when it comes to the problem of edge detection.

\section{Preliminaries}\label{sec2}
Let $d \geq 3$. We consider the Euclidean space $\mathbb{R}^d$ equipped with the usual inner product $\langle x, y \rangle = \sum_{i=1}^d x_i y_i $, as well as the corresponding induced norm $ \| x \|_2 = \sqrt{\langle x, x \rangle}$. The canonical basis vectors of $\mathbb{R}^d$ are given by $e^j=(\delta_{i, j})_{i=1}^d$, $j=1, \dots, d$, where
\begin{equation*}
\delta_{i, j}= \begin{cases}
1, \quad & i=j,\\
0,  & i \neq j.
\end{cases}
\end{equation*}
As usual, the unit sphere in $\mathbb{R}^d$ will be denoted by $\mathbb{S}^{d-1}= \{x \in \mathbb{R}^d \mid \|x\|_2=1 \}$ and we will refer to $e^d$ as the North Pole of $\mathbb{S}^{d-1}$. The standard metric on $\mathbb{S}^{d-1}$ is given by the geodesic distance
\begin{equation*}\label{geodesic distance}
\dist(\eta, \nu)= \arccos\left( \langle \eta, \nu\rangle \right), \qquad \eta, \nu \in \mathbb{S}^{d-1}.
\end{equation*}
The set
\begin{equation*}\label{spherical cap}
C(z, \phi) = \left\{ x\in \mathbb{S}^{d-1} :\dist(x, z)<\phi \right\}   
\end{equation*}
constitutes a spherical cap with center $z \in \mathbb{S}^{d-1}$ and opening angle \(\phi \in \left(0, \pi \right)\). It is sometimes beneficial to parameterize $\mathbb{S}^{d-1}$ in terms of spherical coordinates $\theta_1 \in (-\pi,\pi]$, $\theta_2, \dots, \theta_{d-1}\in [0, \pi]$, via
\begin{equation}\label{spherical coordinates}
\eta(\theta_1, \theta_2, \dots, \theta_{d-1}) = \begin{pmatrix}
\sin \theta_{d-1}\, \dots \, \sin \theta_2 \, \sin \theta_1 \\
\sin \theta_{d-1}\, \dots \, \sin \theta_2 \, \cos \theta_1 \\
\sin \theta_{d-1}\,  \dots\,  \sin \theta_3 \, \cos \theta_2 \\
\vdots \\
\sin \theta_{d-1} \, \cos \theta_{d-2} \\
\cos \theta_{d-1}
\end{pmatrix}.
\end{equation}

In the following, we will review some well-known concepts related to Fourier analysis on $\mathbb{S}^{d-1}$. For more information on this topic, we refer to \cite{bib39, bib3, bib75, bib32}. Let $\omega_{d-1}$ denote the usual rotation-invariant measure on $\mathbb{S}^{d-1}$ which is normalized such that
\begin{equation*}
\int_{\mathbb{S}^{d-1}} \mathrm{d}\omega_{d-1} = 1. 
\end{equation*}
In terms of spherical coordinates \eqref{spherical coordinates},
\begin{equation*}
\mathrm{d}\omega_{d-1} = \frac{\Gamma(\frac{d}{2})}{2 \pi^{d/2}} \sin^{d-2}\theta_{d-1} \, \dots\,\sin \theta_2 \, \mathrm{d}\theta_1 \, \dots\, \mathrm{d}\theta_{d-1}.
\end{equation*}
With respect to this measure, the spaces $L^p(\mathbb{S}^{d-1})$, $0< p < \infty$, are defined by
\begin{equation*}
    L^p(\mathbb{S}^{d-1})= \{ f \colon \mathbb{S}^{d-1} \rightarrow \mathbb{C} \mid f \text{ is measurable and } \|f \|_{L^p(\mathbb{S}^{d-1})} < \infty \},
\end{equation*}
where 
\begin{equation*}
    \|f \|_{L^p(\mathbb{S}^{d-1})} = \left( \int_{\mathbb{S}^{d-1}} \lvert f \rvert^p \, \mathrm{d}\omega_{d-1}\right)^{1/p}.
\end{equation*}
For notational convenience, we will use the symbol $L^\infty(\mathbb{S}^{d-1})$ to refer to the space of continuous functions $f \colon \mathbb{S}^{d-1}\rightarrow \mathbb{C}$, equipped with the supremum norm
\begin{equation*}
    \| f \|_{L^\infty(\mathbb{S}^{d-1})} = \sup_{\eta \in \mathbb{S}^{d-1}} \lvert f(\eta) \rvert.
\end{equation*}

Most investigations conducted in this article rely on Fourier methods which are available in the Hilbert space $L^2(\mathbb{S}^{d-1})$, where the inner product is given by
\begin{equation*}
\langle f_1, f_2 \rangle_{\mathbb{S}^{d-1}} = \int_{\mathbb{S}^{d-1}} f_1  \overline{f_2} \, \mathrm{d}\omega_{d-1}.
\end{equation*}
In particular, we will use a specific orthonormal basis of spherical harmonics (see \cite[p.~466]{bib32}), which has a convenient representation in terms of Gegenbauer polynomials
\begin{equation*}
C_m^\lambda(t) = \frac{(-1)^m \Gamma(\lambda+1/2) \Gamma(m+2\lambda)}{2^m\Gamma(m+\lambda+1/2) m!} (1-t^2)^{1/2-\lambda} \frac{\mathrm{d}^{m}}{\mathrm{d}t^{m}}(1-t^2)^{m+\lambda-1/2}, \quad \lambda >-1/2.
\end{equation*}
W.r.t.\ spherical coordinates \eqref{spherical coordinates}, let
\begin{align}\label{Y_k^n def}
Y_k^{d,n}(\theta_1, \dots, \theta_{d-1})= A_k^n \prod_{j=0}^{d-3} C_{ k_j  - \lvert k_{j+1}\rvert}^{\frac{d-j-2}{2}+\lvert k_{j+1}\rvert}(\cos \theta_{d-j-1}) \, \sin^{\lvert k_{j+1}\rvert}(\theta_{d-j-1}) \, \mathrm{e}^{\mathrm{i}k_{d-2}\theta_1},
\end{align}
where $k_0=n \in \mathbb{N}_0$ and $k=(k_1, \dots, k_{d-2})\in \mathcal{I}_n^d$ with
\begin{equation*}\label{indexset}
\mathcal{I}_n^d = \{ (k_1, \dots, k_{d-2}) \in \mathbb{N}_0^{d-3}\times \mathbb{Z} : n\geq k_1 \geq \dots\geq k_{d-3}\geq \lvert k_{d-2}\rvert \}.
\end{equation*}
The normalization constants $A_k^n >0$ in \eqref{Y_k^n def} are defined by
\begin{align}\label{A_k^n}
 (A_k^n)^2 = \frac{2^{(d-4)(d-2)}}{\Gamma\!\left(\frac{d}{2}\right)} \prod_{j=0}^{d-3}\frac{2^{2\lvert k_{j+1}\rvert-j}(k_j-\lvert k_{j+1}\rvert)!(2k_j+d-j-2)\Gamma^2(\frac{d-j-2}{2}+\lvert k_{j+1}\rvert)}{\sqrt{\pi}\Gamma(k_j+\lvert k_{j+1}\rvert+d-j-2)}.
\end{align}
For $n \in \mathbb{N}_0$ and $k \in \mathcal{I}_n^d$, the function $Y_k^{d, n} \in L^2(\mathbb{S}^{d-1})$ defined above is a spherical harmonic of degree $n$. Moreover, the family $\{ Y_k^{d, n}, \; k \in \mathcal{I}_n^d, \; n \in \mathbb{N}_0 \}$ constitutes an orthonormal basis for $L^2(\mathbb{S}^{d-1})$. The spherical harmonic subspaces $\mathcal{H}_n^d = \vspan \{Y_k^{d, n}, \; k \in \mathcal{I}_n^d \}$ are finite-dimensional with
\begin{equation}\label{dimensions subspaces}
    \dim \mathcal{H}_{n}^d = \frac{(2n+d-2) (n+d-3)!}{(d-2)! n!}.
\end{equation}
By
\begin{equation*}
    \Pi_N(\mathbb{S}^{d-1}) = \bigoplus_{n=0}^N \mathcal{H}_n^d
\end{equation*}
we denote the set of spherical polynomials of degree $N \in \mathbb{N}_0$. With the above notation, every signal $f \in L^2(\mathbb{S}^{d-1})$ can be represented in terms of its Fourier series
\begin{equation*}
    f= \sum_{n=0}^\infty\sum_{k \in \mathcal{I}_n^d} f(n, k) Y_k^{d, n},
\end{equation*}
with the corresponding Fourier coefficients $f(n, k) = \langle f, Y_k^{d, n}\rangle_{\mathbb{S}^{d-1}}$.

Spherical harmonics satisfy the well-known addition theorem  
\begin{equation}\label{addition thm}
\sum_{k \in \mathcal{I}_n^d}  \overline{Y_k^{d, n}(\nu)}  Y_k^{d, n}(\eta) = \frac{2n+d-2}{d-2} C_n^{\frac{d-2}{2}}(\langle \nu, \eta \rangle), \quad \nu, \eta \in \mathbb{S}^{d-1},
\end{equation}
which reveals a connection between harmonic analysis on the sphere $\mathbb{S}^{d-1}$ and on the interval $[-1,1]$. Indeed, for each fixed $\lambda$, the family $\{C_n^\lambda, \; n\in \mathbb{N}_0\}$ forms a complete orthogonal system for the Hilbert space $L^2([-1,1], w_\lambda)$ w.r.t.\ the weighted inner product
\begin{equation*}
    \langle f_1, f_2 \rangle_{[-1,1], w_\lambda} = \int_{-1}^1 f_1(t) \overline{f_2(t)}  w_\lambda(t) \,\mathrm{d}t, \qquad f_1, f_2 \in L^2([-1,1], w_\lambda),
\end{equation*}
where $w_\lambda(t) = (1-t^2)^{\lambda-1/2}$. Here, the Gegenbauer polynomials are normalized such that
\begin{equation*}
\int_{-1}^1 \lvert C_n^\lambda (t) \rvert^2 \, w_\lambda(t) \,  \mathrm{d}t = \frac{\pi \Gamma(n+2\lambda)}{2^{2\lambda-1}n! (n+\lambda) \Gamma^2(\lambda)}.
\end{equation*}
Thus, each $f \in L^2([-1,1], w_\lambda)$ has an expansion of the form
\begin{equation}\label{gegenbauer expansion}
f = \sum_{n=0}^\infty b_n^\lambda(f) C_n^\lambda,
\end{equation}
where
\begin{equation}\label{gegenbauer coeff}
b_n^\lambda(f) = \frac{2^{2\lambda-1}n! (n+\lambda) \Gamma^2(\lambda)}{\pi \Gamma(n+2\lambda)} \int_{-1}^1 f(t)  C_n^\lambda(t) w_\lambda(t) \, \mathrm{d}t.
\end{equation}

Let $SO(d) = \{g \in \mathbb{R}^{d\times d} : g^\top = g^{-1}, \; \det g =1 \}$ be the special orthogonal group. For $g \in SO(d)$, we consider the rotation operator $T^d(g) \colon L^2(\mathbb{S}^{d-1}) \rightarrow L^2(\mathbb{S}^{d-1})$ given by
\begin{equation*}
    T^d(g)f = f(g^{-1} \cdot).
\end{equation*}
In fact, $g \mapsto T^d(g)$ defines a unitary group representation of $SO(d)$ on $L^2(\mathbb{S}^{d-1})$, which is irreducible on each spherical harmonic subspace $\mathcal{H}_n^d$. However, apart from some calculations in \autoref{sec5}, we will not refer to this property explicitly, although it plays an essential role in the background of this article. Indeed, the tools provided by the theory of harmonic analysis associated with the representation $T^d$ constitute the foundation for several central results in \cite{bib77}, on which we will base our constructions. The material required for our investigations in \autoref{sec5} is summarized in the \hyperref[appendix]{Appendix} of this article. For further details on the topic of harmonic analysis on $SO(d)$ associated with the rotation operators $T^d(g)$, $g \in SO(d)$, we refer to \cite{bib32}, as well as to our recent discussions in \cite{bib37, bib77}.

In the following, for $m \in \{2, 3, \dots, d-1 \}$, we will not distinguish between $SO(m)$ and the subgroup $\{ h \in SO(d) : he^j= e^j \text{ for } j=m+1, \dots, d \} \subset SO(d)$. In particular, $SO(d-1)$ is identified with the subgroup of all elements in $SO(d)$ that keep the North Pole $e^d$ fixed. For $\nu \in \mathbb{S}^{d-1}$, the symbol $g_\nu$ will always denote an arbitrary rotation matrix in $SO(d)$ such that $g_\nu e^d = \nu$. Likewise, if $\nu' \in \mathbb{S}^{d-2}$, we write $h_{\nu'}$ to refer to an arbitrary rotation matrix in $SO(d-1)$ satisfying $h_{\nu'} e^{d-1} = (\nu', 0)\in \mathbb{S}^{d-1}$.

Throughout this article, all constants $c, c_0, c_1, \dots$ depend only on the dimension $d$ and on the indicated parameters. Furthermore, their exact values might change with each appearance. For two non-negative sequences $(a_j)_{j=0}^\infty$ and $(b_j)_{j=0}^\infty$, we write $a_j\sim b_j$ to indicate that there exist constants $0<c_1\leq c_2$ such that
\begin{equation*}
    c_1 a_j \leq b_j \leq c_2 a_j, \qquad \text{for each } j \in \mathbb{N}_0.
\end{equation*}

\section{Construction of polynomial curvelets}\label{sec3}
The construction of polynomial frames presented in this chapter is motivated by the article \cite{bib1}, in which the authors defined highly directional analysis functions, so-called second-generation curvelets, on the $2$-sphere. Specifically, utilizing the general framework recently developed in \cite{bib77}, we extend the systems considered in \cite{bib1} to higher-dimensional spheres. We will refer to our frame elements as polynomial curvelets.

Let $\phi \in C^q([0, \infty))$, $q \in \mathbb{N}$, be a real-valued non-increasing function such that $\phi(t)=1$ for $t\in [0, 1/2]$ and $\phi(t)=0$ for $t\geq 1$. Then
\begin{equation*}\label{kappa def}
\kappa(t)=\sqrt{\phi^2(t/2) - \phi^2(t)}, \qquad 0 \leq t < \infty,
\end{equation*}
defines a window function $\kappa \in C^q([0, \infty))$ with $\supp(\kappa)\subset [1/2, 2]$. In particular, it follows that
\begin{equation}\label{admissibility kappa}
    \sum_{j=1}^\infty \left\lvert\kappa\!\left( \frac{n}{2^{j-1}}\right)\right\rvert^2 = 1, \quad \text{for each } n \in \mathbb{N}.
\end{equation}
Additionally, we consider directionality components of the form
\begin{equation*}
\zeta_k^{d, n} = \frac{1}{\sqrt{2}}\begin{cases}
1, \quad &\text{if } \lvert k_{d-2} \rvert = n,\\
0, & \text{else},
\end{cases} \qquad \qquad k=(k_1, \dots . k_{d-2}) \in \mathcal{I}_n^d, \quad n \in \mathbb{N}.
\end{equation*}
Clearly,
\begin{equation}\label{admissibility zeta curvelets}
\sum_{k\in  \mathcal{I}_{n}^d} \lvert \zeta_{k}^{d,n}\rvert^2 = 1, \quad \text{for each }n \in \mathbb{N}.
\end{equation}
With the above defined window function $\kappa$ and the directionality components $\zeta_k^{d, n}$, we now define an initial sequence $(\Psi^j)_{j=0}^{\infty}$ of polynomial curvelets via
\begin{equation}\label{curvelets def1}
    \Psi^0 \equiv 1, \qquad \Psi^{j} = \sum_{n=0}^{\infty}\sum_{k\in  \mathcal{I}_{n}^d} \sqrt{\dim \mathcal{H}_n^d} \, \kappa\!\left(\frac{n}{2^{j-1}} \right) \zeta_{ k}^{d,n} \,  Y_k^{d,n}, \quad j \in \mathbb{N}.
\end{equation}
We note that in the two-dimensional setting, i.e., for $d=3$, this construction yields (essentially) the second-generation curvelets from \cite{bib1}.

By definition of the spherical harmonics given in \eqref{Y_k^n def}, we have
\begin{equation*}
Y_{(n, \dots,n, \pm n)}^{d, n}(x) = A_{(n, \dots, n)}^n (x_2 \pm \mathrm{i}x_1)^n, \qquad \text{for } x = (x_1, x_2, \dots, x_d) \in \mathbb{S}^{d-1}.
\end{equation*}
Thus, the polynomial curvelets defined above can be written as
\begin{align*}
\Psi^{j}(x)  &= \sqrt{2} \sum_{n=0}^{\infty}\sqrt{\dim \mathcal{H}_n^d}  \, \kappa\!\left(\frac{n}{2^{j-1}} \right) A_{(n, \dots, n)}^n \Re\{(x_2+\mathrm{i}x_1)^n\}\\
&=\sqrt{2}\sum_{n=0}^{\infty} \sqrt{\dim \mathcal{H}_n^d}  \, \kappa\!\left(\frac{n}{2^{j-1}} \right) A_{(n, \dots, n)}^n \lvert x_2+\mathrm{i}x_1\rvert^n \cos (n\Arg(x_2+\mathrm{i}x_1)), \qquad j \in \mathbb{N}.
\end{align*}
Here, $\Re\{z\}$ and $\Arg(z)$ denote the real part and the principal argument of $z \in \mathbb{C}$, respectively. To be precise, if $z\neq 0$ then $\Arg(z)$ is the unique element in $(-\pi, \pi]$ such that $z = \lvert z \rvert \exp(\mathrm{i}\Arg(z))$. From this explicit representation, it is obvious that, at large scales $j$, the functions $\Psi^{j}$ are well localized at $e^2$. However, as mentioned in the two-dimensional setting \cite{bib1}, one might prefer to work with an initial sequence where each element is concentrated at the North Pole $e^d$. Hence, for an arbitrary rotation matrix $g_0 \in SO(d)$ with $g_0 e^1 = e^{d-1}$ and $g_0 e^2 = e^d$, we consider the rotated sequence $(\Psi_{\scriptscriptstyle \mathrm{C}}^j)_{j=0}^\infty \subset \Pi(\mathbb{S}^{d-1})$, where
\begin{equation}\label{curvelets def}
\Psi_{\scriptscriptstyle \mathrm{C}}^j = T^d(g_0)\Psi^j, \qquad j \in \mathbb{N}_0.
\end{equation}
Clearly, the polynomial curvelets $\Psi_{\scriptscriptstyle \mathrm{C}}^j$ satisfy
\begin{align}\label{curvelets explicit}
\Psi_{\scriptscriptstyle \mathrm{C}}^j(x)  &= \sqrt{2} \sum_{n=0}^{\infty}\sqrt{\dim \mathcal{H}_n^d}  \, \kappa\!\left(\frac{n}{2^{j-1}} \right) A_{(n, \dots, n)}^n \Re\{(x_d+\mathrm{i}x_{d-1})^n\}\nonumber\\
&=\sqrt{2}\sum_{n=0}^{\infty} \sqrt{\dim \mathcal{H}_n^d}  \, \kappa\!\left(\frac{n}{2^{j-1}} \right) A_{(n, \dots, n)}^n \lvert x_d+\mathrm{i}x_{d-1}\rvert^n \cos (n\Arg(x_{d}+\mathrm{i}x_{d-1})),
\end{align}
which implies that each $\Psi_{\scriptscriptstyle \mathrm{C}}^j$ is concentrated at $e^d$, where the localization increases with $j$. Moreover, \eqref{curvelets explicit} shows that
\begin{equation*}
    T^d(h)\Psi_{\scriptscriptstyle \mathrm{C}}^j = \Psi_{\scriptscriptstyle \mathrm{C}}^j,  \quad \text{for each } h \in SO(d-2),
\end{equation*}
as $\Psi_{\scriptscriptstyle \mathrm{C}}^j(x_1, x_2, \dots, x_d)$ only depends on $x_{d-1}$ and $x_d$. In the terminology of \cite{bib77}, the sequence $(\Psi_{\scriptscriptstyle \mathrm{C}}^j)_{j=0}^\infty$ is $SO(d-2)$-invariant. Hence, as discussed in \cite{bib77}, we can illustrate our polynomial curvelets in terms of the functions
\begin{equation}\label{illustrate curvelets}
\psi_{\scriptscriptstyle \mathrm{C}}^j(t, \varphi) =\Psi_{\scriptscriptstyle \mathrm{C}}^j(\cos t \, e^d + \sin t (\cos \varphi \, e^{d-1} + \sin \varphi \, (0, 0, \eta))), \quad (t, \varphi) \in [0, \pi]\times[0, 2\pi),
\end{equation}
which are independent of $\eta \in \mathbb{S}^{d-3}$. For $\varphi \in [0, 2\pi)$, the map $t \mapsto \psi_{\scriptscriptstyle \mathrm{C}}^j(t, \varphi)$ represents the function $\Psi_{\scriptscriptstyle \mathrm{C}}^j$ along the geodesic (unit-speed) curve $v\colon [0, \pi] \mapsto \mathbb{S}^{d-1}$ with $v(0)=e^d$, $v(\pi)=-e^d$, and $v'(0)= \cos \varphi e^{d-1} + \sin \varphi (0, 0, \eta)$. In particular, since the expression on the right-hand side of \eqref{illustrate curvelets} does not depend on $\eta$, polynomial curvelets $\Psi_{\scriptscriptstyle \mathrm{C}}^j$ exhibit the same behavior along each of those paths $t \mapsto v(t)$, as long as $\varphi$ is fixed. Consequently, we can replace $(0, 0, \eta)$ with $e^{d-2}$ in \eqref{illustrate curvelets}.

\begin{figure}
\centering
\includegraphics[width=\textwidth]{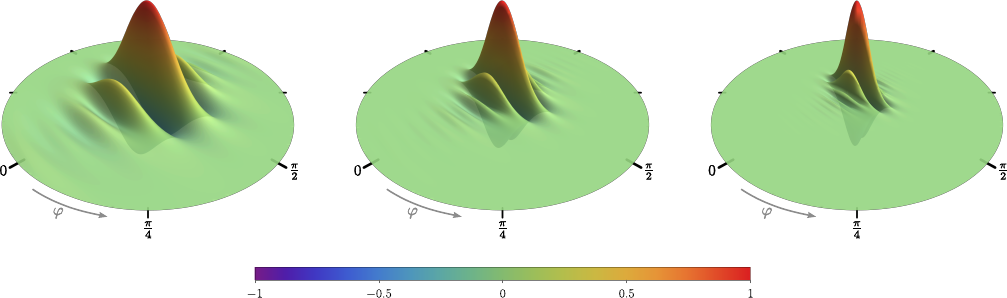}
\caption{Re-scaled polynomial curvelets $\psi_{\scriptscriptstyle \mathrm{C}}^{j} (t, \varphi)$, $(t, \varphi) \in [0, 1]\times[0, 2\pi)$, for $j=5, 6, 7$ from left to right}\label{fig4}
\end{figure}

In \autoref{fig4} the functions $\psi_{\scriptscriptstyle \mathrm{C}}^j$ are visualized for $d=4$, i.e., in the case where $\Psi_{\scriptscriptstyle \mathrm{C}}^j \in \Pi_{2^j}(\mathbb{S}^3)$, where we chose $\kappa$ to be the $C^\infty$ window function with $\supp(\kappa)= [1/2, 2]$ defined in \cite{bib9}. Here, for the sake of clarity, each of the images has been re-scaled to only take values in $[-1, 1]$. In addition to being well localized at the North Pole $e^d$, we observe that the curvelets $\Psi_{\scriptscriptstyle \mathrm{C}}^j$ decay fastest along the directions $\pm e^{d-1}$, which correspond to $\varphi \in \{0, \pi \}$ in \autoref{fig4}, and slowest in the directions $(0,0 , \eta'')$, $\eta'' \in \mathbb{S}^{d-3}$, which correspond to $\varphi \in \{\pi/2, 3\pi/2 \}$.

Following the general approach from \cite{bib77}, the sequence $(\Psi_{\scriptscriptstyle \mathrm{C}}^j)_{j=0}^\infty\subset \Pi(\mathbb{S}^{d-1})$ can be used to generate frames for $L^2(\mathbb{S}^{d-1})$ as follows. For each $j \in \mathbb{N}$, we require positive quadrature formulas of the form
\begin{equation}\label{quadrature S^(d-1)}
    \int_{\mathbb{S}^{d-1}} f \, \mathrm{d}\omega_{d-1} = \sum_{r=1}^{r_j} \omega_{j, r} f(\eta_{j, r}), \qquad \text{if } f \in \Pi_{2^{j+1}}(\mathbb{S}^{d-1}),
\end{equation}
and
\begin{equation}\label{quadrature S^(d-2)}
    \int_{\mathbb{S}^{d-2}} f \, \mathrm{d}\omega_{d-2} = \sum_{s=1}^{s_j} \omega_{j, s}' f(\eta_{j, s}'), \qquad \text{if } f \in \Pi_{2^{j+1}}(\mathbb{S}^{d-2}),
\end{equation}
with points $\eta_{j, r} \in \mathbb{S}^{d-1}$, $\eta_{j, s}' \in \mathbb{S}^{d-2}$ and corresponding weights $\omega_{j, r}, \omega_{j, s}'>0$. For details on spherical quadrature rules and their construction, see \cite{bib34, bib15, bib72, bib73, bib74, bib71} and the references therein. In particular, we can assume (see e.g.\ \cite{bib34}) that the numbers of samples in \eqref{quadrature S^(d-1)} and \eqref{quadrature S^(d-2)} satisfy
\begin{equation}\label{r_ s_j scaling}
    r_j \sim 2^{j(d-1)}, \qquad s_j \sim2^{j(d-2)}.
\end{equation}
As suggested in \cite[Section~4.5.3]{bib77} for the setting of $SO(d-2)$-invariant sequences, we consider the system
\begin{equation}\label{our frame}
    \mathcal{X}[(\Psi_{\scriptscriptstyle \mathrm{C}}^j)_{j=0}^\infty] = \left\{\sqrt{\omega_{j, r} \omega_{j, s}'}\,  T^d(g_{\eta_{j, r}} h_{\eta_{j, s}'})\Psi_{\scriptscriptstyle \mathrm{C}}^j, \; r=1, \dots, r_j,\; s=1, \dots, s_j, \; j \in \mathbb{N}_0 \right\},
\end{equation}
where $r_0 = s_0 =1$, and where $\eta_{0, 1}\in \mathbb{S}^{d-1}$, $\eta_{0, 1}'\in \mathbb{S}^{d-2}$ can be chosen arbitrarily. Here, using the shorthand notation
\begin{equation*}
    \Psi_{\scriptscriptstyle \mathrm{C}}^{j, r, s} = \sqrt{\omega_{j, r} \omega_{j, s}'}\,  T^d(g_{\eta_{j, r}} h_{\eta_{j, s}'})\Psi_{\scriptscriptstyle \mathrm{C}}^j, \quad r=1, \dots, r_j,\; s=1, \dots, s_j, \; j \in \mathbb{N}_0,
\end{equation*}
we identify $\mathcal{X}[(\Psi_{\scriptscriptstyle \mathrm{C}}^j)_{j=0}^\infty]$ with the (canonically) ordered sequence
\begin{equation*}
    ( \Psi_{\scriptscriptstyle \mathrm{C}}^{0, 1, 1}, \Psi_{\scriptscriptstyle \mathrm{C}}^{1, 1, 1}, \dots , \Psi_{\scriptscriptstyle \mathrm{C}}^{1, 1, s_1}, \Psi_{\scriptscriptstyle \mathrm{C}}^{1, 2, 1} \dots, \Psi_{\scriptscriptstyle \mathrm{C}}^{1, 2, s_1}, \dots \dots, \Psi_{\scriptscriptstyle \mathrm{C}}^{1, r_1, s_1}, \Psi_{\scriptscriptstyle \mathrm{C}}^{2, 1, 1}, \dots \dots ).
\end{equation*}
We note that, according to \eqref{r_ s_j scaling}, the number of rotations at each scale $j$ in \eqref{our frame} grows like $2^{j(2d-3)}$. In comparison, for previous constructions such as spherical needlets and directional polynomial wavelets, this number typically scales like $2^{j(d-1)}$, which can be regarded as a manifestation of their limited directional sensitivity (see \cite{bib77} for more details).
\begin{proposition}\label{proposition curvelets tight frame}
The system $\mathcal{X}[(\Psi_{\scriptscriptstyle \mathrm{C}}^j)_{j=0}^\infty]$ is a Parseval frame for $L^2(\mathbb{S}^{d-1})$. In particular,
\begin{equation*}
    f = \sum_{j=0}^\infty \sum_{r=1}^{r_j} \sum_{s=1}^{s_j}  \langle f,  \Psi_{\scriptscriptstyle \mathrm{C}}^{j, r, s} \rangle_{\mathbb{S}^{d-1}}  \Psi_{\scriptscriptstyle \mathrm{C}}^{j, r, s}
\end{equation*}
converges unconditionally in $L^2(\mathbb{S}^{d-1})$, for each $f \in L^2(\mathbb{S}^{d-1})$.
\end{proposition}
\begin{proof}
According to \cite[Proposition~3.1]{bib77} it suffices to show that
\begin{equation}\label{another admissibility}
   \delta_{n, 0} + (\dim \mathcal{H}_n^d)^{-1}\sum_{j=1}^\infty \sum_{k \in \mathcal{I}_n^d} \lvert  \Psi_{\scriptscriptstyle \mathrm{C}}^{j}(n, k)\rvert^2 = 1, \qquad \text{for each } n \in \mathbb{N}_0.
\end{equation}
By construction, we have $\Psi_{\scriptscriptstyle \mathrm{C}}^{j} =T^d(g_0)\Psi^j$, $j \in \mathbb{N}_0$, where $\Psi^j$ is defined by \eqref{curvelets def1}. Therefore, $\Psi^{j}(n, k)$ and $\Psi_{\scriptscriptstyle \mathrm{C}}^{j}(n, k) = \langle \Psi^j, T^d(g_0^{-1})Y_k^{d, n}\rangle_{\mathbb{S}^{d-1}}$ both represent Fourier coefficients of $\Psi^j$ w.r.t.\ the orthonormal bases $\{ Y_k^{d, n}, \; k \in \mathcal{I}_n^d\}$ and $\{ T^d(g_0^{-1})Y_k^{d, n}, \; k \in \mathcal{I}_n^d\}$ of $\mathcal{H}_n^d$. Consequently,
\begin{equation*}
    \sum_{k \in \mathcal{I}_n^d}\lvert \Psi_{\scriptscriptstyle \mathrm{C}}^{j}(n, k)\rvert^2 = \sum_{k \in \mathcal{I}_n^d}\lvert \Psi^{j}(n, k)\rvert^2
\end{equation*}
and thus \eqref{another admissibility} reduces to
\begin{equation*}
   \delta_{n, 0} + \sum_{j=1}^\infty \left\lvert \kappa\!\left(\frac{n}{2^{j-1}} \right) \right\rvert^2 \sum_{k \in \mathcal{I}_n^d} \lvert \zeta_{k}^{d, n}\rvert^2 = 1, \qquad \text{for each } n \in \mathbb{N}_0.
\end{equation*}
The assertion now follows easily from \eqref{admissibility kappa} and \eqref{admissibility zeta curvelets}.
\end{proof}

To conclude this section, we note that the curvelet expansion converges fast in the spaces $L^p(\mathbb{S}^{d-1})$, $1 \leq p \leq \infty$, where we recall that, by our convention, the symbol $L^\infty (\mathbb{S}^{d-1})$ refers to the Banach space of continuous functions on $\mathbb{S}^{d-1}$. For each $J \in \mathbb{N}_0$, let $\Lambda_J \colon L^1(\mathbb{S}^{d-1}) \rightarrow \Pi_{2^{J}}(\mathbb{S}^{d-1})$ be the linear operator defined by
\begin{equation*}
    \Lambda_J f =  \sum_{j=0}^J \sum_{r=1}^{r_j} \sum_{s=1}^{s_j}  \langle f,  \Psi_{\scriptscriptstyle \mathrm{C}}^{j, r, s} \rangle_{\mathbb{S}^{d-1}}  \Psi_{\scriptscriptstyle \mathrm{C}}^{j, r, s}.
\end{equation*}
Also, let
\begin{equation*}
E_m(f)_p = \inf_{P \in \Pi_m(\mathbb{S}^{d-1})} \| f-P \|_{L^p(\mathbb{S}^{d-1})}, \qquad m \in \mathbb{N}_0, \quad 1\leq p \leq \infty,
\end{equation*}
denote the error of the best approximation of $f \in L^p(\mathbb{S}^{d-1})$ in $\Pi_{m}(\mathbb{S}^{d-1})$ w.r.t.\ $\| \cdot \|_{L^p(\mathbb{S}^{d-1})}$. Then we have the following result, where we assume, for the sake of simplicity, that the function $\phi$ utilized in the beginning of the section is infinitely often differentiable, i.e., $\phi \in C^\infty([0, \infty))$.
\begin{proposition}
Let $1 \leq p \leq \infty$. Then
    \begin{equation*}
        \| f - \Lambda_Jf\|_{L^p(\mathbb{S}^{d-1})} \leq c_p E_{2^{J-1}}(f)_p, \qquad \text{for each } f \in L^p(\mathbb{S}^{d-1}).
    \end{equation*}
\end{proposition}
\begin{proof}
    Exactly as in the proof of \cite[Lemma~5.1]{bib37}, a simple calculation shows that
    \begin{equation*}
        \Lambda_J f = \sum_{n=0}^\infty \sum_{k \in \mathcal{I}_n^d}   \phi^2\!\left( \frac{n}{2^J}\right)  \langle f, Y_k^{d, n} \rangle_{\mathbb{S}^{d-1}} Y_k^{d, n}.
    \end{equation*}
    In other words, $\Lambda_J f = f \ast\Phi^J$, where 
    \begin{equation*}
        f \ast \Phi^J (\eta) = \langle f, T^d(g_\eta) \Phi^J\rangle_{\mathbb{S}^{d-1}}, \qquad \eta \in \mathbb{S}^{d-1},
    \end{equation*}
denotes the standard spherical convolution of $f$ with the zonal scaling function
\begin{equation}\label{kernel zonal}
    \Phi^J = \sum_{n=0}^\infty \sum_{k \in \mathcal{I}_n^d} \phi^2\!\left( \frac{n}{2^J}\right) \overline{Y_{k}^{d, n}(e^d)}\,  Y_k^{d, n}.
\end{equation}
For such a kernel \eqref{kernel zonal}, it is well known that $\| f \ast \Phi^J\|_{L^p(\mathbb{S}^{d-1})} \leq c_p \|f \|_{L^p(\mathbb{S}^{d-1})}$. A proof of the latter can be found in \cite[Theorem~2.6.3]{bib3}.

It is easy to see that $\Lambda_Jf = f$ for each $f \in \Pi_{2^{J-1}}(\mathbb{S}^{d-1})$. Thus, if $P \in \Pi_{2^{J-1}}(\mathbb{S}^{d-1})$ is the best approximation of $f$ in $\Pi_{2^{J-1}}(\mathbb{S}^{d-1})$ w.r.t.\ $\|\cdot \|_{L^p(\mathbb{S}^{d-1})}$, then 
\begin{equation*}
     \| f - \Lambda_Jf\|_{L^p(\mathbb{S}^{d-1})} \leq  \| f - P\|_{L^p(\mathbb{S}^{d-1})}+  \| \Lambda_J(P - f)\|_{L^p(\mathbb{S}^{d-1})} \leq c_p E_{2^{J-1}}(f)_p.
\end{equation*}
\end{proof}

\section{Localization}\label{sec4}
In the following, we derive a localization bound for polynomial curvelets. Here, in contrast to the established estimates \eqref{localization bound isotr. introd} for spherical needlets \cite{bib15, bib16, bib3, bib81} and directional polynomial wavelets \cite{bib37, bib31}, the upper bound exhibits a direction-dependent decay, which indicates a particularly strong directional sensitivity. We note that, for $d=3$, where our frames essentially coincide with the second-generation curvelets from \cite{bib1}, the following localization estimate was previously established in \cite[Proposition~2.4]{bib26}.

\begin{theorem}\label{theorem curvelets localization}
Let $ \Psi_{\scriptscriptstyle \mathrm{C}}^{j}$, $j \in \mathbb{N}_0$, be as defined in \autoref{sec3} with $\kappa \in C^q([0, \infty))$ for some $q \in \mathbb{N}$. We assume that there exists some $t_0 \in (1/2, 2)$ such that $\kappa^{(q)}(t) \neq 0$ for each $t \in (1/2, t_0)$. Then
\begin{equation*}\label{localization version 1}
\lvert \Psi_{\scriptscriptstyle \mathrm{C}}^{j}(x)\rvert \leq \frac{c_{q} 2^{j(3d-2)/4} \lvert x_d+\mathrm{i}x_{d-1}\rvert^{2^{j-2}}}{(1+2^j \lvert \Arg(x_d+\mathrm{i}x_{d-1}) \rvert)^q}, \qquad x=(x_1, \dots, x_d) \in \mathbb{S}^{d-1}.
\end{equation*}
Here, $c_q>0$ is a constant which only depends on $\kappa, q$, and $d$.
\end{theorem}
\begin{proof}
Our proof uses arguments similar to those in the two-dimensional setting \cite[Proposition~2.4]{bib26}. We start from the representation given in \eqref{curvelets explicit}, i.e.,
\begin{align*}
    \Psi_{\scriptscriptstyle \mathrm{C}}^j(x) =\sqrt{2}\sum_{n=0}^{\infty} \sqrt{\dim \mathcal{H}_n^d}  \, \kappa\!\left(\frac{n}{2^{j-1}} \right) A_{(n, \dots, n)}^n \lvert x_d+\mathrm{i}x_{d-1}\rvert^n \cos (n\Arg(x_{d}+\mathrm{i}x_{d-1})).
\end{align*}
Utilizing the duplication formula for the Gamma function, the normalization constants $A_{(n, \dots, n)}^n$ defined by \eqref{A_k^n} can be written in the form  
\begin{equation}\label{A_n^n formula}
A_{(n, \dots, n)}^n = \frac{2^{(2-d)/2}}{\sqrt{\Gamma(\frac{d}{2})}} \prod_{j=0}^{d-3}\sqrt{\frac{(2n+d-j-2) \Gamma(n+\frac{d-j-2}{2})}{\Gamma(n+\frac{d-j-1}{2})}}.
\end{equation}
With \eqref{dimensions subspaces} and \eqref{A_n^n formula}, it is easy to verify that there exist constants $c_0, c_1, \dots, c_{q-1}$ such that
\begin{equation*}
 \sqrt{\dim \mathcal{H}_n^d}\, A_{(n, \dots, n)}^n = n^{3(d-2)/4}\left(\sum_{p=0}^{q-1}c_p n^{-p}  + \mathcal{O}(n^{-q})\right), \qquad n \rightarrow \infty.
\end{equation*}
Consequently,
\begin{align*}
& \Psi_{\scriptscriptstyle \mathrm{C}}^j(x)  = R_j(x) + \sum_{p=0}^{q-1} c_p \sum_{n=0}^\infty \kappa\!\left(\frac{n}{2^{j-1}} \right) n^{3(d-2)/4-p} \lvert x_d+\mathrm{i}x_{d-1}\rvert^n \cos (n\Arg(x_d+\mathrm{i}x_{d-1}))
\end{align*}
where, due to the fact that $\supp(\kappa) \subset [1/2, 2]$,
\begin{equation*}
\lvert R_j(x) \rvert \leq c_q 2^{j((3d-2)/4-q)}\lvert x_d+\mathrm{i}x_{d-1}\rvert^{2^{j-2}}.
\end{equation*}
To complete the proof, it suffices to show that
\begin{align*}
\left\lvert \sum_{n=0}^\infty \kappa\!\left(\frac{n}{N} \right) n^{3(d-2)/4-p} z^n  \cos(n\varphi)\right\rvert \leq \frac{c_p z^{N/2}N^{(3d-2)/4}}{(1+N \varphi )^q},
\end{align*}
for each $N \in \mathbb{N}$, $z \in [0, 1]$, and $\varphi \in [0, \pi]$. In order to do so, we first observe that
\begin{equation*}
\kappa\!\left(\frac{n}{N} \right) n^{3(d-2)/4-p} z^n = N^{3(d-2)/4-p}h_{z, N}\!\left(\frac{n}{N}\right),
\end{equation*}
with $h_{z, N}(t) = \kappa(t) t^{3(d-2)/4-p} z^{Nt}$. Clearly, $h_{z, N} \in C^q([0, \infty))$ with $\supp(h_{z, N}) \subset [1/2, 2]$. Thus, it suffices to prove the estimate
\begin{equation}\label{key estimate}
\left\lvert \frac{1}{N} \sum_{n=0}^\infty h_{z, N}\!\left(\frac{n}{N} \right)   \cos(n\varphi)\right\rvert \leq \frac{c_q z^{N/2}}{(1+N \varphi )^q}.
\end{equation}
Applying \cite[Lemma~5]{bib14}, we have
\begin{align}\label{eq34}
\left\lvert \frac{1}{N} \sum_{n=0}^\infty h_{z, N}\!\left(\frac{n}{N} \right)   \cos(n\varphi)\right\rvert \leq \left\lvert \int_0^2h_{z, N}(t) \cos(N t \varphi) \, \mathrm{d}t \right\rvert + \frac{c_q V(h_{z, N}^{(q-1)})}{N^q},
\end{align}
where
\begin{equation*}
V(h_{z, N}^{(q-1)}) = \int_0^2\lvert h_{z, N}^{(q)}(t)\rvert \, \mathrm{d}t, \qquad  h_{z, N}^{(q)}(t) = \frac{\mathrm{d}^q}{\mathrm{d }t^q} h_{z, N}(t).
\end{equation*}
Here, $V(h_{z, N}^{(q-1)}) \leq c_q z^{N/2}$, which can be verified as follows. By Leibniz's rule, $h_{z, N}^{(q)}(t)$ is a linear combination of functions of the form
\begin{equation*}
\kappa^{(q-\ell)}(t) \frac{\mathrm{d}^{\ell-k}}{\mathrm{d}t^{\ell-k}}\left( t^{3(d-2)/4-p}\right) (\ln z^N)^k z^{Nt}, \qquad 0 \leq k\leq \ell\leq q.
\end{equation*}
Thus, it is sufficient to prove the estimate
\begin{equation}\label{eq33}
\lvert \ln(z^N)\rvert^k \int_{1/2}^2 z^{Nt} \lvert \kappa^{(q-\ell)}(t)\rvert \, \mathrm{d}t \leq c_q z^{N/2}.
\end{equation}
By definition of $\kappa$, we have $\kappa^{(m)}(1/2) = \kappa^{(m)}(2)=0$ for $0 \leq m \leq q$. Also, by assumption there exists some $t_0 \in (1/2, 2)$ such that $\kappa^{(q)}$ does not vanish in $(1/2, t_0)$. Hence, the mean value theorem implies that $\kappa^{(m)}(t)\neq 0$, for all $t \in (1/2, t_0)$, $0 \leq m \leq q$. It follows that
\begin{align*}
\int_{1/2}^2 z^{Nt} \lvert \kappa^{(q-\ell)}(t)\rvert \, \mathrm{d}t = \left\lvert \int_{1/2}^{t_0} z^{Nt} \kappa^{(q-\ell)}(t) \, \mathrm{d}t \right\rvert + \int_{t_0}^2 z^{Nt} \lvert \kappa^{(q-\ell)}(t)\rvert \, \mathrm{d}t.
\end{align*}
Repeated integration by parts yields
\begin{align*}
& \int_{1/2}^{t_0} z^{Nt} \kappa^{(q-\ell)}(t) \, \mathrm{d}t  = z^{Nt_0}\sum_{r=0}^{k-1}\frac{\kappa^{(q-\ell+1)}(t_0)}{(\ln z^N)^{r+1}} + \frac{(-1)^k}{(\ln z^N)^k}\int_{1/2}^{t_0} z^{N t} \kappa^{(q-\ell+k)}(t) \, \mathrm{d}t.
\end{align*}
Thus, since $s^\delta \ln(s)^m \rightarrow 0$ as $s \rightarrow 0^+$, for all $\delta >0$ and all $m \in \mathbb{N}$, it is now easy to see that \eqref{eq33} holds. Consequently, $V(h_{z, N}^{(q-1)}) \leq c_q z^{N/2}$, as claimed. Indeed, by the same argument we find that $V(h_{z, N}^{(m)}) \leq c_q z^{N/2}$, for each $0 \leq m \leq q-1$.

Lastly, we derive an estimate for the integral on the right hand sight of \eqref{eq34}. Here, applying integration by parts repeatedly and using the fact that $V(h_{z, N}^{(m)}) \leq c_q z^{N/2}$, for $0 \leq m \leq q-1$, we obtain
\begin{align*}
\left\lvert \int_0^2h_{z, N}(t) \cos(N t \varphi) \, \mathrm{d}t \right\rvert  &\leq \left\lvert \int_0^2h_{z, N}(t) \mathrm{e}^{\mathrm{i} N t \varphi} \, \mathrm{d}t \right\rvert\\
&=\left\lvert \int_0^2h_{z, N}(t) \mathrm{e}^{- \mathrm{i}t}\mathrm{e}^{\mathrm{i}  t(1+N \varphi)} \, \mathrm{d}t \right\rvert \\
& \leq \frac{c_q z^{N/2}}{(1+N \varphi)^q}.
\end{align*}
This completes the proof.
\end{proof}

\autoref{theorem curvelets localization} can be utilized to study the behavior of $\|\Psi_{\scriptscriptstyle \mathrm{C}}^{j} \|_{L^p(\mathbb{S}^{d-1})} $ as $j$ becomes large. Precisely, we have the following result.

\begin{corollary}\label{corollary norm estimate curvelet}
Let $ \Psi_{\scriptscriptstyle \mathrm{C}}^{j}$, $j \in \mathbb{N}_0$, be as defined in \autoref{sec3}, with the same assumptions on $\kappa \in C^q([0, \infty))$ as in \autoref{theorem curvelets localization}. If $0<p<\infty$ and $q > 1/p$, then
\begin{equation*}
\|\Psi_{\scriptscriptstyle \mathrm{C}}^{j} \|_{L^p(\mathbb{S}^{d-1})} \sim 2^{j((3d-2)/4 -d/(2p))}.
\end{equation*}
Moreover, (as long as $q \geq 1$) it holds that
\begin{equation*}
\|\Psi_{\scriptscriptstyle \mathrm{C}}^{j} \|_{L^\infty(\mathbb{S}^{d-1})} = \lvert \Psi_{\scriptscriptstyle \mathrm{C}}^{j}(e^d) \rvert \sim 2^{j(3d-2)/4}.
\end{equation*}
Here, the constants of equivalence only depend on $\kappa, q, p$, and $d$.
\end{corollary}
\begin{proof}
We closely follow the arguments for the zonal setting given in \cite{bib30}. In some instances, it will be convenient to work with the functions $\Psi^{j} = T(g_0^{-1})\Psi_{\scriptscriptstyle \mathrm{C}}^{j}$ given by \eqref{curvelets def1}, rather than with the curvelets $\Psi_{\scriptscriptstyle \mathrm{C}}^{j}$ themselves. Of course, the latter does not pose any problems due to the rotational invariance of the surface measure $\omega_{d-1}$ on $\mathbb{S}^{d-1}$.

For $p=2$, the representation \eqref{curvelets def1} immediately yields
\begin{equation*}
\|\Psi^{j} \|_{L^2(\mathbb{S}^{d-1})}^2 = \sum_{n=0}^{\infty} \dim \mathcal{H}_n^d \left\lvert \kappa\!\left( \frac{n}{2^{j-1}} \right) \right\rvert^2.
\end{equation*}
Thus, using \eqref{dimensions subspaces}, it is not difficult to verify that $\|\Psi^j \|_{L^2(\mathbb{S}^{d-1})} \sim 2^{j(d-1)/2}$. 

In the general case $p\in (0, \infty)$, we can utilize \hyperref[theorem curvelets localization]{Theorem~\ref*{theorem curvelets localization}}, which, in terms of spherical coordinates \eqref{spherical coordinates}, implies the estimate
\begin{equation*}
\lvert \Psi^{j}(\theta_1, \dots, \theta_{d-1})\rvert \leq \frac{c_{q} \left(\prod_{m=2}^{d-1}\sin \theta_{m} \right)^{2^{j-2}} 2^{j(3d-2)/4}}{(1+2^j \lvert \theta_1\rvert)^q}.
\end{equation*}
With this estimate, a straightforward calculation yields 
\begin{equation*}
\|\Psi^{j} \|_{L^p(\mathbb{S}^{d-1})}^p \leq c_q 2^{j(p(3d-2)/4 -1)}\left( \int_{0}^{\pi/2} (\sin \theta)^{2^{j-2}p} \, \mathrm{d}\theta\right)^{d-2},
\end{equation*}
provided that $q > 1/p$. By the well-known formula
\begin{equation*}
\int_0^{\pi/2} \sin^m \theta \, \mathrm{d}\theta = \frac{\sqrt{\pi}\Gamma(\frac{m+1}{2})}{2 \Gamma(\frac{m+2}{2})} \leq c m^{-1/2},
\end{equation*}
it now follows that
\begin{equation*}
\|\Psi^{j} \|_{L^p(\mathbb{S}^{d-1})} \leq c 2^{j((3d-2)/4-d/(2p))}.
\end{equation*}

The estimate $\|\Psi_{\scriptscriptstyle \mathrm{C}}^{j} \|_{L^{\infty}(\mathbb{S}^{d-1})} = \lvert\Psi_{\scriptscriptstyle \mathrm{C}}^{j} (e^d) \rvert \leq c2^{j(3d-2)/4} $ is a direct consequence of \eqref{curvelets explicit} and \hyperref[theorem curvelets localization]{Theorem~\ref*{theorem curvelets localization}}. It can be used to obtain the corresponding lower bound in the case $0<p<2$, since
\begin{equation*}
\|\Psi_{\scriptscriptstyle \mathrm{C}}^{j} \|_{L^2(\mathbb{S}^{d-1})}^2 \leq \|\Psi_{\scriptscriptstyle \mathrm{C}}^{j} \|_{L^{\infty}(\mathbb{S}^{d-1})}^{2-p} \|\Psi_{\scriptscriptstyle \mathrm{C}}^{j} \|_{L^p(\mathbb{S}^{d-1})}^p.
\end{equation*}

For $2<p\leq \infty$, the lower estimate follows from the above and from Hölder's inequality. Precisely, we have
\begin{equation*}
\|\Psi_{\scriptscriptstyle \mathrm{C}}^{j} \|_{L^2(\mathbb{S}^{d-1})}^2 \leq \|\Psi_{\scriptscriptstyle \mathrm{C}}^{j} \|_{L^p(\mathbb{S}^{d-1})} \|\Psi_{\scriptscriptstyle \mathrm{C}}^{j}\|_{L^q(\mathbb{S}^{d-1})}, \quad q = \frac{p}{p-1},
\end{equation*}
as well as
\begin{equation*}
\|\Psi_{\scriptscriptstyle \mathrm{C}}^{j} \|_{L^2(\mathbb{S}^{d-1})}^2 \leq \|\Psi_{\scriptscriptstyle \mathrm{C}}^{j} \|_{L^\infty(\mathbb{S}^{d-1})}  \|\Psi_{\scriptscriptstyle \mathrm{C}}^{j} \|_{L^1(\mathbb{S}^{d-1})}.
\end{equation*}
This completes the proof.
\end{proof}

We note that the asymptotic behavior 
\begin{equation}\label{norm behavior 1}
\|\Psi_{\scriptscriptstyle \mathrm{C}}^{j} \|_{L^p(\mathbb{S}^{d-1})} \sim 2^{j((3d-2)/4 -d/(2p))}, \qquad 0<p \leq \infty,
\end{equation}
established in \hyperref[corollary norm estimate curvelet]{Corollary~\ref{corollary norm estimate curvelet}}, deviates from the typical scaling of the $L^p$-norms for previously considered polynomial frames. Specifically, if $(\tilde{\Psi}^j)_{j=0}^\infty$ represents an initial sequence of spherical needlets or directional polynomial wavelets on $\mathbb{S}^{d-1}$, then, according to \cite{bib30} and \cite{bib37}, 
\begin{equation}\label{norm behavior 2}
\|\tilde{\Psi}^{j} \|_{L^p(\mathbb{S}^{d-1})} \sim 2^{j(d-1)(1-1/p)}, \qquad 0<p \leq \infty.
\end{equation}
Clearly, \eqref{norm behavior 1} and \eqref{norm behavior 2} only coincide in the case $p=2$. 

\section{Directional sensitivity}\label{sec5}
For the moment, let $(\Psi^j)_{j=0}^\infty$ be an arbitrary sequence of spherical polynomials $\Psi^j \in \Pi_{2^j}(\mathbb{S}^{d-1})$. We assume that each of these polynomials is (in some sense) concentrated at the North Pole $e^d$. In this setting, the functions
\begin{equation*}
    T^d(h)\Psi^j, \qquad h \in SO(d-1),
\end{equation*}
are also concentrated at $e^d$, and we refer to them as orientations of $\Psi^j$. As proposed in the two-dimensional setting \cite{bib21}, the directionality of $\Psi^j$ can be measured using the auto-correlation function
\begin{equation*}
   SO(d-1) \rightarrow\mathbb{C}, \qquad h \mapsto  \langle T^d(h)\Psi^j, \Psi^j \rangle_{\mathbb{S}^{d-1}},
\end{equation*}
which tests $\Psi^j$ against its various orientations. In this context, a high directional sensitivity of $\Psi^j$ is associated with an auto-correlation function that is well localized at the identity matrix $\mathrm{I}_d \in SO(d-1)$.

Previous constructions of localized polynomial frames generated by an initial sequence $(\Psi^j)_{j=0}^\infty$, namely spherical needlets \cite{bib11, bib3, bib15, bib16, bib29, bib30, bib14, bib46, bib48, bib49, bib25}, directional wavelets \cite{bib21, bib10, bib9, bib37}, and second-generation curvelets \cite{bib1, bib26}, differ vastly in their degree of directionality.  Specifically, if $(\Psi_{\scriptscriptstyle\mathrm{N}}^j)_{j=0}^\infty$ denotes an initial sequence of spherical needlets, then
\begin{equation*}
    \langle T^d(h)\Psi_{\scriptscriptstyle\mathrm{N}}^j, \Psi_{\scriptscriptstyle\mathrm{N}}^j  \rangle_{\mathbb{S}^{d-1}} = \| \Psi_{\scriptscriptstyle\mathrm{N}}^j\|_{L^2(\mathbb{S}^{d-1})}^2, \qquad \text{for each } h \in SO(d-1).
\end{equation*}
Hence, the corresponding auto-correlation functions are constant and the associated frames do not exhibit any directionality whatsoever. If $(\Psi_{\scriptscriptstyle \textrm{W}, K}^j)_{j=0}^\infty$ denotes an initial sequence of directional wavelets on $\mathbb{S}^2$, as defined in \cite{bib21, bib9}, where $K\in \mathbb{N}_0$ is some fixed parameter controlling the degree of anisotropy, then
\begin{equation}\label{auto-corr wav 2}
     \langle T^3(h(\gamma))\Psi_{\scriptscriptstyle \textrm{W}, K}^j, \Psi_{\scriptscriptstyle \textrm{W}, K}^j  \rangle_{\mathbb{S}^{2}} = \sum_{n=0}^\infty \dim \mathcal{H}_n^3 \left\lvert \kappa\!\left(\frac{n}{2^{j-1}} \right) \right\rvert^2 \cos^{K_n}\gamma, \qquad \gamma \in [0, 2\pi),
\end{equation}
in which $K_n = \min(K, n)$ and $h(\gamma)\in \mathbb{R}^{3\times 3}$ is a positive rotation by $\gamma$ in the $(x_1, x_2)$-plane. Recently, we introduced directional polynomial wavelets $\Psi_{\scriptscriptstyle \textrm{W}, K}^j \in \Pi_{2^j}( \mathbb{S}^{d-1})$ for higher-dimensional spheres (see \cite{bib37}). Here, we have shown that
\begin{equation}\label{auto-corr wave d-1}
    \langle T^d(h) \Psi_{\scriptscriptstyle \textrm{W}, K}^j,  \Psi_{\scriptscriptstyle \textrm{W}, K}^j\rangle_{\mathbb{S}^{d-1}}= \sum_{n=0}^{\infty}\dim \mathcal{H}_n^d  \left\lvert \kappa\!\left(\frac{n}{2^{j-1}} \right) \right\rvert^2 (\langle e^{d-1} , h e^{d-1}\rangle)^{K_n}, \qquad h \in SO(d-1).
\end{equation}

The formulas \eqref{auto-corr wav 2} and \eqref{auto-corr wave d-1} imply that, in contrast to spherical needlets, directional polynomial wavelets can exhibit a strong degree of directional sensitivity, provided that $K$ is chosen large enough. However, since $K$ is a parameter that must be fixed a priori, the directionality does not increase with $j$. Frames with this property are also referred to as steerable frames (see \cite{bib64, bib21, bib9, bib10, bib37, bib77}).

Among the previously considered constructions of localized polynomial frames, second-generation curvelets on $\mathbb{S}^2$, as introduced in \cite{bib1}, stand out because they are unlimited in their directional resolution. Indeed, this property carries over to the higher-dimensional polynomial curvelets defined in \autoref{sec3}. Specifically, we have the following result.
\begin{proposition}\label{prop curvelet autocorr}
    Let $ \Psi_{\scriptscriptstyle \mathrm{C}}^{j}$ be as defined in \autoref{sec3}. For $d=3$,
    \begin{equation}\label{autocorr d=3 curv}
        \langle T^3(h(\gamma))  \Psi_{\scriptscriptstyle \mathrm{C}}^{j},  \Psi_{\scriptscriptstyle \mathrm{C}}^{j}\rangle_{\mathbb{S}^{2}} = \sum_{n=0}^\infty \dim \mathcal{H}_n^3  \left\lvert \kappa\!\left( \frac{n}{2^{j-1}} \right)\right\rvert^2 \!\left(\cos^{2n}\frac{\gamma}{2} + \sin^{2n}\frac{\gamma}{2} \right), \quad \gamma \in [0, 2\pi),
    \end{equation}
    where $h(\gamma)\in \mathbb{R}^{3\times3}$ is a positive rotation by $\gamma$ in the $(x_1,x_2)$-plane. If $d \geq 4$, then
    \begin{align}\label{autocorr d>3 curv}
        &\langle T^d(h)  \Psi_{\scriptscriptstyle \mathrm{C}}^{j},  \Psi_{\scriptscriptstyle \mathrm{C}}^{j}\rangle_{\mathbb{S}^{d-1}}\nonumber \\
        &\qquad  =2\sum_{n=0}^\infty \dim \mathcal{H}_n^d \left\lvert \kappa\!\left( \frac{n}{2^{j-1}} \right)\right\rvert^2 \sum_{\substack{m=0 \\ m \text{ even}}}^n\lvert c(d, n, m)\rvert^2 \frac{\Gamma(d-3) m!}{\Gamma(m+d-3)} C_m^{\frac{d-3}{2}}(\langle h e^{d-1}, e^{d-1}\rangle),
    \end{align}
    for each $h \in SO(d-1)$, where
    \begin{equation*}
        c(d, n, m) = A_{(m, 0, \dots, 0)}^n A_{(n, \dots, n)}^n \frac{ \sqrt{\pi}\,\Gamma(\frac{d}{2}) \binom{n}{m} \Gamma(m+d-3) (-\mathrm{i})^m}{2^{m+d-4}(n+\frac{d-2}{2})\Gamma(\frac{d-2}{2}) \Gamma(\frac{d-3}{2})\Gamma(m+\frac{d-2}{2})}.
    \end{equation*}
\end{proposition}
\begin{proof}
For $d=3$, the analysis functions $\Psi_{\scriptscriptstyle \mathrm{C}}^{j}$ (essentially) coincide with the second-generation curvelets from \cite{bib1}, and \eqref{autocorr d=3 curv} follows from \cite[Proposition~2.3]{bib26}.

Now, we consider the case $d \geq 4$. The following proof relies on fundamental properties of the matrix functions associated with the unitary group representation $T^d$. These properties are summarized in the \hyperref[appendix]{Appendix} of this article. By \eqref{curvelets def}, we have
    \begin{equation*}
       \Psi_{\scriptscriptstyle \mathrm{C}}^{j} = \frac{1}{\sqrt{2}}\sum_{n=0}^\infty \sqrt{\dim \mathcal{H}_n^d} \,\kappa\!\left( \frac{n}{2^{j-1}} \right)  \left(T^d(g_0) Y_{(n, \dots, n)}^{d, n} + T^d(g_0)Y_{(n, \dots, n)^-}^{d, n} \right),
    \end{equation*}
    with $(n, \dots, n, n)^- = (n, \dots, n, -n)$. Hence, it follows from the definition of the matrix functions in \eqref{matrix fct def} that
    \begin{equation*}
         \langle T^d(h)  \Psi_{\scriptscriptstyle \mathrm{C}}^{j},  \Psi_{\scriptscriptstyle \mathrm{C}}^{j}\rangle_{\mathbb{S}^{d-1}} = \sum_{n=0}^\infty \dim \mathcal{H}_n^d  \left\lvert \kappa\!\left( \frac{n}{2^{j-1}} \right)\right\rvert^2 L_n(h),
    \end{equation*}
    where
    \begin{align}\label{matrixfct curvelets auto}
        & L_n(h)= \frac{1}{2} \bigg(t_{(n, \dots, n), (n, \dots, n)}^{d, n}(g_0^{-1} h g_0) +  t_{(n, \dots, n)^-, (n, \dots, n)}^{d, n}(g_0^{-1} h g_0) \nonumber \\
        & \qquad\qquad\qquad\qquad  + t_{(n, \dots, n), (n, \dots, n)^-}^{d, n}(g_0^{-1} h g_0) + t_{(n, \dots, n)^-, (n, \dots, n)^-}^{d, n}(g_0^{-1} h g_0)\bigg).
    \end{align}
    Here, using \eqref{matrix fct addition formula} and \eqref{matrix fct restrict}, we obtain
    \begin{equation}\label{involved matrix fct}
        t_{(n, \dots, n), (n, \dots, n)}^{d, n}(g_0^{-1} h g_0) = \sum_{m=0}^n\sum_{k, \ell \in \mathcal{I}_{m}^{d-1}} t_{(n, \dots, n), (m, k)}^{d, n}(g_0^{-1}) \,  t_{k, \ell}^{d-1, m}(h) t_{(m, \ell), (n, \dots, n)}^{d, n}(g_0).
    \end{equation}
    Similar expressions hold for the other matrix functions in \eqref{matrixfct curvelets auto}.

    Clearly, $L_n$ is an element of the polynomial subspace $\mathcal{M}_n^{1}(SO(d-1))$ defined in \eqref{class 1 matrix fct}. Also, by \eqref{curvelets explicit}, $\Psi_{\scriptscriptstyle \mathrm{C}}^{j}$ is invariant under rotations that keep $e^d$ and $e^{d-1}$ fixed, i.e.,
    \begin{equation}\label{invariance eq18k}
        T^d(k)\Psi_{\scriptscriptstyle \mathrm{C}}^{j} = \Psi_{\scriptscriptstyle \mathrm{C}}^{j}, \qquad \text{for each } k \in SO(d-2).
    \end{equation}
    Indeed, \eqref{invariance eq18k} is a direct consequence of the $SO(d-2)$-invariance of the spherical harmonics $T^d(g_0)Y_{(n, \dots, n)}^{d, n}$ and $T^d(g_0)Y_{(n, \dots, n)^-}^{d, n}$. It follows that $L_n$ exhibits the symmetry
    \begin{equation}\label{symmetry L_n}
        L_n(k_1 hk_2) = L_n(h), \qquad \text{for } k_1, k_2 \in SO(d-2), \; h \in SO(d-1).
    \end{equation}
    
    The symmetry \eqref{symmetry L_n} implies that $L_n$ is a linear combination of the matrix functions $t_{0, 0}^{d-1, m}$, as we will now demonstrate. Let $m \in \mathbb{N}_0$ and $r, s \in \mathcal{I}_m^{d-1}$. By \eqref{integral onion} and \eqref{matrix fct addition formula}, we have
    \begin{align*}
        \langle L_n, t_{r, s}^{d-1, m} \rangle_{SO(d-1)} &= \int_{\mathbb{S}^{d-2}}  \int_{SO(d-2)} L_n(h_{\eta} k) \, \overline{t_{r, s}^{d-1, m}(h_{\eta} k)} \, \mathrm{d}\mu_{d-2}(k) \, \mathrm{d}\omega_{d-2}(\eta) \\
        &=\int_{\mathbb{S}^{d-2}} L_n (h_{\eta}) \int_{SO(d-2)} \overline{\sum_{q \in \mathcal{I}_m^{d-1}} t_{r, q}^{d-1, m}(h_{\eta})  t_{q, s}^{d-1, m}(k)} \, \mathrm{d}\mu_{d-2}(k) \, \mathrm{d}\omega_{d-2}(\eta).
    \end{align*}
    According to \eqref{matrix fct restrict}, it holds that $t_{q, s}^{d-1, m}(k) = \delta_{q_1, s_1} t_{(q_2, \dots, q_{d-2}), (s_2, \dots, s_{d-2})}^{d-2, s_1}(k)$, for each $k \in SO(d-2)$ and $d \geq 5$. For $d=4$, \eqref{matrix fct restrict d=3} yields $t_{q, s}^{3, m}(k(\gamma)) = \delta_{q, s} \mathrm{e}^{\mathrm{i}q\gamma}$, where $k(\gamma)\in \mathbb{R}^{3\times 3}$ is a positive rotation by $\gamma$ in the $(x_1, x_2)$-plane. Thus, due to the orthogonality of the matrix functions and due to the fact that $t_{0, 0}^{d-2, 0} \equiv 1$, for $d \geq 5$, we obtain
    \begin{equation*}
        \langle L_n, t_{r, s}^{d-1, m} \rangle_{SO(d-1)} = \delta_{s, 0} \langle L_n, t_{r,0}^{d-1, m} \rangle_{SO(d-1)}.
    \end{equation*}
    At this point, we confirmed that $L_n$ can be written as a linear combination of the matrix functions $t_{r, 0}^{d-1, m}$, with $r \in \mathcal{I}_m^{d-1}$, $0 \leq m \leq n$. In particular, using \eqref{matrix fct and spherical harmonics}, the function $F_n \colon\mathbb{S}^{d-2} \rightarrow \mathbb{C}$ defined by $F_n(\eta) = L_n(h_\eta)$ constitutes a polynomial of degree $n$ on $\mathbb{S}^{d-2}$, i.e., $F_n \in \Pi_n(\mathbb{S}^{d-2})$. By \eqref{symmetry L_n},
    \begin{equation*}
        F_n(k\eta) = L_n(k h_\eta)= L_n(h_\eta) = F_n(\eta), \qquad \text{for each } k \in SO(d-2), \; \eta \in \mathbb{S}^{d-2}.
    \end{equation*}
    Since $Y_0^{d-1, m}$ is the unique element in $\mathcal{H}_m^{d-1}$, up to multiplication by a scalar, that is invariant under the group $SO(d-2)$ (see e.g.\ \cite[Lemma~1.7.1]{bib3}), it follows that
    \begin{equation*}
        F_n = \sum_{m=0}^n \langle F_n, Y_0^{d-1, m}\rangle_{\mathbb{S}^{d-2}} Y_0^{d-1, m}.
    \end{equation*}
    Hence, by the relation $L_n(h) = F_n(he^{d-1})$ and by \eqref{matrix fct and spherical harmonics}, we obtain
    \begin{equation}\label{L_n property}
        \langle L_n, t_{r, s}^{d-1, m} \rangle_{SO(d-1)} = \delta_{s, 0} \delta_{r, 0} \langle L_n, t_{0, 0}^{d-1, m} \rangle_{SO(d-1)}.
    \end{equation}

    Now, we consider the representation of $L_n$ given in \eqref{matrixfct curvelets auto}. By expanding the involved matrix functions as in \eqref{involved matrix fct}, and by using \eqref{L_n property}, it follows that
    \begin{align*}
        & L_n(h) = \frac{1}{2}\sum_{m=0}^n t_{0, 0}^{d-1, m}(h) \big(t_{(n, \dots, n), (m, 0, \dots, 0)}^{d, n}(g_0^{-1}) t_{(m, 0, \dots, 0), (n, \dots, n)}^{d, n}(g_0) \\
        &\quad + t_{(n, \dots, n)^-, (m, 0, \dots, 0)}^{d, n}(g_0^{-1}) t_{(m, 0, \dots, 0), (n, \dots, n)}^{d, n}(g_0) + t_{(n, \dots, n), (m, 0, \dots, 0)}^{d, n}(g_0^{-1}) t_{(m, 0, \dots, 0), (n, \dots, n)^-}^{d, n}(g_0)\\
        & \quad + t_{(n, \dots, n)^-, (m, 0, \dots, 0)}^{d, n}(g_0^{-1}) t_{(m, 0, \dots, 0), (n, \dots, n)^-}^{d, n}(g_0) \big) .
    \end{align*}
    Moreover, the symmetries \eqref{symmetry matrix fct 1} and \eqref{symmetry matrix fct 2} yield
    \begin{align*}
         &L_n(h) = \frac{1}{2}\sum_{m=0}^n t_{0, 0}^{d-1, m}(h)\bigg( \left\lvert t_{(n, \dots, n), (m, 0, \dots, 0)}^{d, n}(g_0^{-1})\right\rvert^2 + \left\lvert t_{(n, \dots, n)^-, (m, 0, \dots, 0)}^{d, n}(g_0^{-1})\right\rvert^2 \\
         & \qquad + 2 \Re\left\{ \big(  t_{(n, \dots, n)^-, (m, 0, \dots, 0)}^{d, n}(g_0^{-1})\big)^2\right\} \bigg).
    \end{align*}
    Here, using \eqref{matrix fct and spherical harmonics}, it is not difficult to verify that
    \begin{equation*}
        t_{0, 0}^{d-1, m}(h) = \frac{\Gamma(d-3) m!}{\Gamma(m+d-3)}C_m^{\frac{d-3}{2}}(\langle h e^{d-1}, e^{d-1}\rangle).
    \end{equation*}
    For an explicit reference of this formula, we refer to \cite[p.~471]{bib32}. 
    
    In order to prove \eqref{autocorr d>3 curv}, it remains to show that
    \begin{equation}\label{matrix fct c(d,n,m)}
        t_{(n, \dots, n), (m, 0, \dots, 0)}^{d, n}(g_0^{-1}) = c(d, n, m).
    \end{equation}
    We have
    \begin{align*}
         t_{(n, \dots, n), (m, 0, \dots, 0)}^{d, n}(g_0^{-1}) &= \int_{\mathbb{S}^{d-1}} Y_{(m, 0, \dots, 0)}^{d, n}(x) \, \overline{T^d(g_0) Y_{(n, \dots, n)}^{d, n}(x)} \, \mathrm{d}\omega_{d-1}(x) \\
         &= A_{(m, 0, \dots, 0)}^n A_{(n, \dots, n)}^n \int_{\mathbb{S}^{d-1}} C_{n-m}^{\frac{d-2}{2}+m}(x_d) (1-x_d^2)^{m/2}\\
         & \quad \times C_m^{\frac{d-3}{2}}\!\bigg( \frac{x_{d-1}}{\sqrt{\smash[b]{1-x_d^2}}}\bigg) (x_d - \mathrm{i}x_{d-1})^n \, \mathrm{d}\omega_{d-1}(x).
    \end{align*}
    In terms of spherical coordinates \eqref{spherical coordinates}, 
    \begin{align*}
        &\int_{\mathbb{S}^{d-1}} C_{n-m}^{\frac{d-2}{2}+m}(x_d) (1-x_d^2)^{m/2}C_m^{\frac{d-3}{2}}\!\bigg( \frac{x_{d-1}}{\sqrt{\smash[b]{ 1-x_d^2}}}\bigg) (x_d - \mathrm{i}x_{d-1})^n \, \mathrm{d}\omega_{d-1}(x)\\
        &\qquad \qquad \qquad  \qquad = c(d) \int_0^\pi C_{n-m}^{\frac{d-2}{2}+m}(\cos \theta_{d-1}) H_{n, m}(\theta_{d-1}) (\sin\theta_{d-1})^{m+d-2} \, \mathrm{d}\theta_{d-1},
    \end{align*}
    where $c(d)= \Gamma(\frac{d}{2})/(\pi \Gamma(\frac{d-2}{2}))$ and
    \begin{equation*}
        H_{n,m}(\theta_{d-1}) = \int_0^\pi C_m^{\frac{d-3}{2}}(\cos\theta_{d-1}) (\cos \theta_{d-1} - \mathrm{i} \cos \theta_{d-2} \sin \theta_{d-1})^n \sin^{d-3} \theta_{d-2}\, \mathrm{d}\theta_{d-2}.
    \end{equation*}
    A well-known representation formula for Gegenbauer polynomials (see e.g.\ \cite[p.~483]{bib32}) yields
    \begin{equation*}
        H_{n, m}(\theta) = \frac{2^{m+1} \sqrt{\pi}\, \Gamma(\frac{d-2}{2}+m) n! \Gamma(m+d-3)}{m!\Gamma(\frac{d-3}{2}) \Gamma(n+m+d-2) \mathrm{i}^m} \sin^m \theta \,  C_{n-m}^{\frac{d-2}{2}+m}(\cos \theta).
    \end{equation*}
    Furthermore, as shown in \cite[p.~462]{bib32},
    \begin{equation*}
        \int_0^\pi \lvert C_{n-m}^{\frac{d-2}{2}+m}(\cos \theta)\rvert^2 (\sin \theta)^{2m+d-2} \, \mathrm{d}\theta = \frac{\pi \Gamma(n+m+d-2)}{2^{2m+d-3} (n-m)! (n+\frac{d-2}{2} ) \Gamma^2(m+ \frac{d-2}{2})}.
    \end{equation*}
    By combining all of the above, formula \eqref{matrix fct c(d,n,m)} can be verified, and thus the proof of \eqref{autocorr d>3 curv} is complete.

\end{proof}

\begin{figure}
\centering
\includegraphics[width=0.7\textwidth]{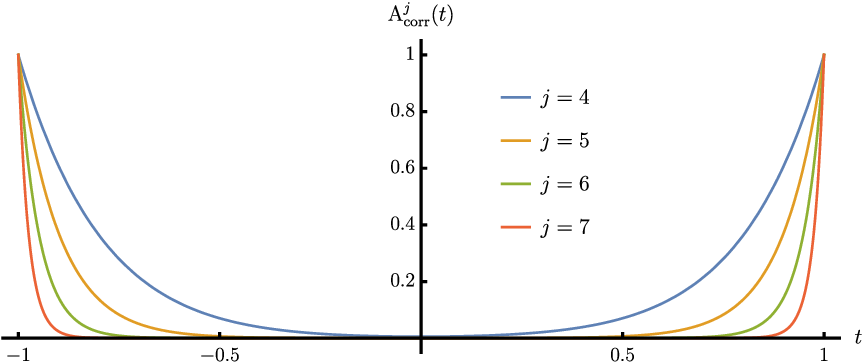}
\caption{Auto-correlation function $\mathrm{A}_{\mathrm{corr}}^j$ for $j=4, 5, 6, 7$}\label{figAutoCorr}
\end{figure}

In \autoref{figAutoCorr}, the normalized auto-correlation functions
\begin{equation*}\label{A_corr}
    \mathrm{A}_{\mathrm{corr}}^j(t)=\frac{\langle T^d(h)  \Psi_{\scriptscriptstyle \mathrm{C}}^{j},  \Psi_{\scriptscriptstyle \mathrm{C}}^{j}\rangle_{\mathbb{S}^{d-1}}}{\|    \Psi_{\scriptscriptstyle \mathrm{C}}^{j}\|_{L^2(\mathbb{S}^{d-1})}^2} , \qquad t=\langle h e^{d-1}, e^{d-1}\rangle,
\end{equation*}
are visualized for $d=4$, i.e., in the case where $ \Psi_{\scriptscriptstyle \mathrm{C}}^{j} \in \Pi_{2^j}(\mathbb{S}^3)$, and for different scales $j$. Here, the representation of 
$\langle T^d(h)  \Psi_{\scriptscriptstyle \mathrm{C}}^{j},  \Psi_{\scriptscriptstyle \mathrm{C}}^{j}\rangle_{\mathbb{S}^{d-1}}$ as a function of $t=\langle h e^{d-1}, e^{d-1}\rangle$ is given by \eqref{autocorr d>3 curv} and we chose $\kappa$ to be the $C^\infty$ window with $\supp (\kappa) = [1/2, 2]$ constructed in \cite{bib9}. 

\section{Detection of singularities}\label{sec6}
To conclude this article, we demonstrate the potential of polynomial curvelets for applications that require localized analyses of anisotropic features. Here, we focus on the problem of identifying edges and higher-order singularities by the asymptotic decay of the curvelet coefficients. Specifically, we will derive a higher-dimensional version of a previous result for the two-dimensional setting \cite[Theorem~3.3]{bib26}.

For $z \in \mathbb{S}^{d-1}$, $r \in (0, \pi)$, and $\tau \in \mathbb{N}_0$, we consider the isotropic/zonal test signal $f_{z, r, \tau} \colon \mathbb{S}^{d-1}\rightarrow \mathbb{R}$ given by
\begin{equation}\label{def test signal}
f_{z, r, \tau}(\eta)=  \tilde{f}_{\cos r, \tau}(\langle \eta, z\rangle), \quad \tilde{f}_{t_0, \tau}(t):=
\begin{cases}
( t-t_0)^\tau, \quad & t\in [t_0, 1],\\
0, & t \in [-1,t_0 ).
\end{cases}
\end{equation}
In particular, $f_{z, r, \tau}$ is smooth except for a discontinuity in the $r$-th derivative at the boundary $\partial C(z, r)$ of the spherical cap $C(z, r)$. For $\tau=0$, the test signal is just the characteristic function of $C(z, r)$, i.e., $f_{z, r, 0} = \mathbf{1}_{C(z, r)}$. In the following, we will assume that $z=e^d$. Also, we will use the simplified notation $f_{  r, \tau}:=f_{e^d, r, \tau} $.

According to the discussion following \autoref{fig4} in \autoref{sec3}, the curvelet $ \Psi_{\scriptscriptstyle \mathrm{C}}^{j}$ is localized at $e^d$ and decays fastest along the directions $\pm e^{d-1}$. Consequently, for $\nu \in \mathbb{S}^{d-1}$, the rotated curvelet $T^d(g_\nu^{-1}) \Psi_{\scriptscriptstyle \mathrm{C}}^{j}$ is localized at $g_\nu^{-1}e^{d}$ and decays fastest along the directions $\pm g_{\nu}^{-1} e^{d-1}$. It is reasonable to expect that the frame coefficients $\lvert \langle f_{r, \tau}, T^d(g_\nu^{-1})\Psi_{\scriptscriptstyle \mathrm{C}}^{j} \rangle_{\mathbb{S}^{d-1}}\rvert$ take their largest values when the curvelet $T^d(g_\nu^{-1}) \Psi_{\scriptscriptstyle \mathrm{C}}^{j}$ matches the boundary $\partial C(e^d, r)$ sufficiently well in terms of position and orientation. Here, we note that $T^d(g_\nu^{-1}) \Psi_{\scriptscriptstyle \mathrm{C}}^{j}$ is positioned exactly on a (higher-order) singularity of $f_{ r, \tau}$ if and only if
\begin{equation}\label{orientation match 0}
\langle g_\nu^{-1}e^d, e^d \rangle = \nu_d = \cos r.
\end{equation}
Moreover, provided that $ \nu \neq \pm e^d$, the curvelet $T^d(g_\nu^{-1}) \Psi_{\scriptscriptstyle \mathrm{C}}^{j}$ aligns exactly with the orientation of the boundary $\partial C(e^d, r)$ whenever
\begin{equation}\label{orientation match 1}
g_{\nu}^{-1} e^{d-1}= \pm\frac{e^d-\langle e^d, g_\nu^{-1}e^d\rangle g_\nu^{-1}e^d}{\| e^d-\langle e^d, g_\nu^{-1}e^d\rangle g_\nu^{-1}e^d \|_2}.
\end{equation}
Multiplying with $g_\nu$ from the left on both sides of \eqref{orientation match 1}, a simple calculation shows that \eqref{orientation match 1} is equivalent to
\begin{equation}\label{orientation match 2}
   \nu_{d-1}^2 + \nu_{d}^2 = 1, \qquad \text{i.e., } \quad  \nu = (0, \dots, 0, \nu_{d-1}, \nu_d).
\end{equation}
Thus, combining \eqref{orientation match 0} and \eqref{orientation match 2}, the curvelet $T^d(g_\nu^{-1}) \Psi_{\scriptscriptstyle \mathrm{C}}^{j}$ matches some part of the boundary $\partial C(e^d, r)$ perfectly, in terms of position and orientation, if
\begin{equation*}
    \nu = (0, \dots, 0, \pm \sin r, \cos r).
\end{equation*}
In the following, we will show that the curvelet coefficients $\langle f_{r, \tau}, T^d(g_\nu^{-1})\Psi_{\scriptscriptstyle \mathrm{C}}^{j} \rangle_{\mathbb{S}^{d-1}}$ peak when $\nu$ is sufficiently close to $(0, \dots, 0, \pm \sin r, \cos r)$. Otherwise, the inner products decay rapidly as $j$ increases. 

\subsection{Some auxiliary results}
Before we present the main theorems of this section, we need to establish some auxiliary results. The following lemma provides useful asymptotic formulas for the harmonic coefficients of the test signal defined in \eqref{def test signal}.
\begin{lemma}\label{lemma testsignal coeff}
    Let $\lambda = \frac{d-2}{2}$. It holds that
    \begin{equation}\label{test signal fourier coeff}
        \langle f_{r, \tau}, Y_k^{d, n} \rangle_{\mathbb{S}^{d-1}} =  \delta_{k, 0} \, b_n^\lambda(\tilde{f}_{\cos r, \tau}) \sqrt{\frac{\Gamma(n+d-2) (d-2)}{n! \Gamma(d-2) (2n+d-2)}}.
    \end{equation}
    Here, the Gegenbauer coefficients $ b_n^\lambda(\tilde{f}_{\cos r, \tau})$ are defined by \eqref{gegenbauer coeff}. Moreover, if
    \begin{equation*}
    \gamma(\lambda, r, \tau) = \frac{4^\lambda \Gamma(\lambda)^2 \Gamma(\lambda + \frac{1}{2}) \sqrt{\lambda} \tau! (\sin r)^{\lambda + \tau}}{2^{\tau + 1/2}\pi^{3/2} \Gamma(2\lambda)^{3/2}},
    \end{equation*}
    then, for each $\varepsilon>0$, the asymptotic expressions
    \begin{align}\label{test signal coeff asympt1}
    &\langle f_{r, \tau}, Y_0^{d, n} \rangle_{\mathbb{S}^{d-1}} = \gamma(\lambda, r, \tau) \, n^{-\tau-1} \cos ( (n+\lambda)r - (\lambda+\tau+1)\pi/2) + \mathcal{O}(n^{-\tau-2}),
    \end{align}
    $n \rightarrow \infty$, and 
    \begin{equation}\label{test signal coeff asympt2}
    \langle f_{r, \tau}, Y_0^{d, n} \rangle_{\mathbb{S}^{d-1}} = \Re\left\{ \mathrm{e}^{\mathrm{i}n r} \sum_{p=0}^{q-1} \phi_p(r) n^{-\tau-p-1}\right\} + \mathcal{O}(n^{-\tau-q-1}), \qquad n \rightarrow \infty,
    \end{equation}
    hold uniformly for all $r \in [\varepsilon, \pi-\varepsilon]$. Here, $\phi_p$, $p=0, \dots, q-1$, are continuous complex-valued functions.
\end{lemma}
\begin{proof}
Writing down the Gegenbauer expansion \eqref{gegenbauer expansion} of $\tilde{f}_{\cos r, \tau} \in L^2([-1,1],w_\lambda )$ and utilizing the addition theorem for spherical harmonics, we find that
\begin{equation*}
    f_{z, r, \tau}(\eta) = \sum_{n=0}^\infty b_n^\lambda(\tilde{f}_{\cos r, \tau}) C_n^{\lambda}(\langle \eta, z \rangle) = \sum_{n=0}^\infty\sum_{k \in \mathcal{I}_n^d} b_n^\lambda(\tilde{f}_{\cos r, \tau}) \frac{ \lambda}{n+\lambda}  \overline{ Y_k^{d,n}(z)} Y_k^{d,n}(\eta).
\end{equation*}
Thus, for $z=e^d$, formula \eqref{test signal fourier coeff} follows immediately from 
\begin{equation*}
    Y_k^{d, n}(e^d) = \delta_{k, 0} \sqrt{\frac{\Gamma(n+d-2) (2n+d-2)}{n! \Gamma(d-1)}},
\end{equation*}
which can be verified using \eqref{Y_k^n def} and \eqref{A_k^n}.

As discussed in \cite[Lemma~4]{bib14}, it holds that
\begin{align}\label{test signal gegenbauer coeff}
&\int_{\cos r}^1 (t-\cos r)^\tau \, C_n^{\lambda}(t) \,w_\lambda(t) \, \mathrm{d}t =\frac{\Gamma(n+2\lambda)\Gamma(\lambda +1/2)}{\Gamma(2\lambda) \Gamma(n+\lambda+1/2)}\nonumber\\
& \qquad \qquad \qquad  \qquad \times \frac{(n-\tau-1)! \tau!}{2^{\tau+1} n!} P_{n-\tau-1}^{(\lambda + \tau + 1/2, \lambda + \tau + 1/2)}(\cos r)\, (\sin r)^{2\lambda + 2\tau+1},
\end{align}
as well as
\begin{align}\label{eq38a}
P_{n-\tau-1}^{(\lambda + \tau + 1/2, \lambda + \tau + 1/2)}(\cos r) &= \sqrt{\frac{2}{\pi (n-\tau-1)}} (\sin r)^{-\lambda-\tau-1} \cos( (n+\lambda)r - (\lambda+\tau+1)\pi/2)  \nonumber \\
&\quad + \mathcal{O}(n^{-3/2}), \qquad n \rightarrow \infty,
\end{align}
and
\begin{equation}\label{eq38b}
P_{n-\tau-1}^{(\lambda + \tau + 1/2, \lambda + \tau + 1/2)}(\cos r)  =  \Re \left\{ \mathrm{e}^{\mathrm{i}nr} \sum_{p=0}^{q-1} \phi_p(r) n^{-p-1/2}\right\} + \mathcal{O}(n^{-q-1/2}), \qquad n \rightarrow \infty.
\end{equation}
Here, the $\phi_p$ are continuous complex-valued functions and both asymptotic expressions hold uniformly for all $r \in [\delta, \pi-\delta]$, whenever $\delta >0$. Combining \eqref{test signal fourier coeff} and \eqref{test signal gegenbauer coeff}, we obtain
\begin{align}\label{eq38c}
&\langle f_{r, \tau}, Y_k^{d, n} \rangle_{\mathbb{S}^{d-1}}  =  \delta_{k, 0} \frac{4^\lambda \sqrt{\lambda}\,\Gamma^2(\lambda) \Gamma(\lambda + \frac{1}{2}) \tau! (\sin r)^{2\lambda+2\tau+1}}{2^{\tau+1} \pi \Gamma(2 \lambda)^{3/2}}\nonumber \\
& \qquad \qquad \qquad \quad \quad \times  \frac{\sqrt{\Gamma(n+2\lambda)(n+\lambda)} (n-\tau-1)!}{\sqrt{n!}\Gamma(n+\lambda+\frac{1}{2})} P_{n-\tau-1}^{(\lambda + \tau + 1/2, \lambda + \tau + 1/2)}(\cos r).
\end{align}
Finally, inserting \eqref{eq38a} and \eqref{eq38b} into \eqref{eq38c}, it is straightforward to verify the asymptotic formulas \eqref{test signal coeff asympt1} and \eqref{test signal coeff asympt2}.
\end{proof}

Another helpful harmonic representation is given by the following lemma.
\begin{lemma}\label{lemma rep of rot curv}
For $\nu \in \mathbb{S}^{d-1}$, it holds that
\begin{align}\label{rotated curvelet harmonic rep}
\langle T^d(g_\nu^{-1})\Psi_{\scriptscriptstyle \mathrm{C}}^{j}, Y_0^{d, n} \rangle_{\mathbb{S}^{d-1}} &= \sqrt{2}\, \kappa\!\left( \frac{n}{2^{j-1}}\right)A_{(n, \dots, n)}^n \lvert \nu_d + \mathrm{i}\nu_{d-1}\rvert^n \cos(n \Arg(\nu_d + \mathrm{i}\nu_{d-1})),
\end{align}
where
\begin{equation}\label{A_n^n asympt} 
A_{(n, \dots, n)}^n = \frac{n^{(d-2)/4}}{\sqrt{\Gamma(\frac{d}{2})}} \left(\sum_{p=0}^{q-1} c_p n^{-p} + \mathcal{O}(n^{-q})   \right), \qquad n \rightarrow \infty,
\end{equation}
with $c_0=1$.
\end{lemma}
\begin{proof}
By \eqref{curvelets explicit}, we have
\begin{align*}
\Psi_{\scriptscriptstyle \mathrm{C}}^j(x)  &= \sqrt{2} \sum_{n=0}^{\infty}\sqrt{\dim \mathcal{H}_n^d}  \, \kappa\!\left(\frac{n}{2^{j-1}} \right) A_{(n, \dots, n)}^n H_n(x),
\end{align*}
where $H_n(x) = \Re\{(x_d+\mathrm{i}x_{d-1})^n\}$, for $x \in \mathbb{S}^{d-1}$. Since, by construction, $H_n$ is a multiple of $Y_{(n, \dots, n)}^{d, n} + Y_{(n, \dots, n)^-}^{d, n}$, it is clear that $H_n \in \mathcal{H}_n^d$. Thus, $T^d(g_\nu^{-1})H_n \in \mathcal{H}_n^d$, for each $\nu \in \mathbb{S}^{d-1}$.

In order to verify \eqref{rotated curvelet harmonic rep}, it suffices to show that
\begin{equation*}
     \sqrt{\dim \mathcal{H}_n^d} \,\langle T^d(g_\nu^{-1})H_n, Y_0^{d, n} \rangle_{\mathbb{S}^{d-1}} = \Re\{(\nu_d+\mathrm{i}\nu_{d-1})^n\}.
\end{equation*}
Here, we first note that
\begin{equation*}
       \sqrt{\dim \mathcal{H}_n^d}\,  Y_0^{d, n}(x) =\frac{2n+d-2}{d-2} C_n^{\frac{d-2}{2}}(\langle x, e^d\rangle).
\end{equation*}
Also, the map
\begin{equation*}
      (x, y)\mapsto  \frac{2n+d-2}{d-2} C_n^{\frac{d-2}{2}}(\langle x, y\rangle)
\end{equation*}
represents the reproducing kernel of $\mathcal{H}_n^d$, which is a direct consequence of the addition theorem \eqref{addition thm}. It follows that
\begin{equation*}
     \sqrt{\dim \mathcal{H}_n^d} \,\langle T^d(g_\nu^{-1})H_n, Y_0^{d, n} \rangle_{\mathbb{S}^{d-1}} = T^d(g_\nu^{-1})H_n(e^d) = H_n(\nu).
\end{equation*}
This completes the proof of \eqref{rotated curvelet harmonic rep}.

Finally, the asymptotic expression \eqref{A_n^n asympt} can easily be derived from \eqref{A_n^n formula}.
\end{proof}

\subsection{Main results}
For our first main result, we assume that the rotated curvelet $T^d(g_\nu^{-1})\Psi_{\scriptscriptstyle \mathrm{C}}^{j}$ matches the orientation of the boundary $\partial C(e^d, r)$ perfectly. According to the discussion in the beginning of this section, this means that $\nu_{d-1}^2 + \nu_d^2=1$. In this case, the following theorem provides valuable insight into the behavior of the frame coefficients $\langle f_{r, \tau}, T^d(g_\nu^{-1})\Psi_{\scriptscriptstyle \mathrm{C}}^{j} \rangle_{\mathbb{S}^{d-1}}$ when $T^d(g_\nu^{-1})\Psi_{\scriptscriptstyle \mathrm{C}}^{j}$ moves towards, and away from, a (higher-order) singularity of $f_{r, \tau}$. In particular, we obtain a lower bound which states that the coefficients exceed a certain threshold whenever the curvelet is positioned close enough (depending on the scale $j$) to the boundary $\partial C(e^d, r)$. 

We note that, in the case where $\nu_{d-1}^2+\nu_d^2=1$, the geodesic distance between the North Pole and the center of $T^d(g_\nu^{-1})\Psi_{\scriptscriptstyle \mathrm{C}}^{j}$ takes the form
\begin{equation*}
    \dist (e^d,g_\nu^{-1}e^d ) = \arccos \nu_d = \lvert \Arg(\nu_d +\mathrm{i} \nu_{d-1})\rvert.
\end{equation*}
Hence, the quantity $\lvert \Arg(\nu_d +\mathrm{i} \nu_{d-1})\rvert-r$ indicates how far the curvelet $T^d(g_\nu^{-1})\Psi_{\scriptscriptstyle \mathrm{C}}^{j}$ is from the nearest (higher-order) singularity. 

The following theorem is inspired by a one-dimensional result from \cite[Theorem~2]{bib14}, and it was previously obtained in \cite[Theorem~3.3]{bib26} for the two-dimensional case.

\begin{theorem}\label{thm detect 1}
Let $ \varepsilon >0$, $\lambda = \frac{d-2}{2}$, $\varphi = (\lambda+\tau+1)\pi/2 - \lambda r$, and let
\begin{equation*}
    \rho(\lambda, r, \tau) = \frac{(2\pi)^{\tau-(d-2)/4}\gamma(\lambda, r, \tau)}{ \sqrt{2\Gamma(\frac{d}{2})}},
\end{equation*}
where $\gamma(\lambda, r, \tau )$ is as defined in \hyperref[lemma testsignal coeff]{Lemma~\ref{lemma testsignal coeff}}. Additionally, let $\nu \in \mathbb{S}^{d-1}$ with $\nu_d^2 + \nu_{d-1}^2 = 1$ and let $r \in [\varepsilon, \pi-\varepsilon]$.
There exists a constant $c_0>0$ such that 
\begin{align}\label{asympt formula main}
&2^{j(\tau-(d-2)/4)}\langle f_{r, \tau}, T^d(g_\nu^{-1})\Psi_{\scriptscriptstyle \mathrm{C}}^{j} \rangle_{\mathbb{S}^{d-1}} \nonumber\\
&\quad  = \rho(\lambda, r, \tau)  \int_{0}^{2\pi} \kappa(t/\pi)t^{(d-2)/4 - \tau-1} \cos\!\left(\frac{2^{j} t (\lvert \Arg(\nu_d + \mathrm{i}\nu_{d-1})\rvert -r )}{2\pi} + \varphi\right)  \mathrm{d}t + \mathrm{E}^j(\nu), 
\end{align}
where
\begin{equation*}
   \lvert \mathrm{E}^j(\nu) \rvert \leq c_02^{-j}.
\end{equation*}
Moreover, there exist constants $c_1, j_0>0$ as well as some nonempty open interval $I \subset [0, \infty)$ such that
\begin{equation}\label{lower bound main}
2^{j(\tau-(d-2)/4)} \lvert \langle f_{r, \tau}, T^d(g_\nu^{-1})\Psi_{\scriptscriptstyle \mathrm{C}}^{j} \rangle_{\mathbb{S}^{d-1}} \rvert \geq c_1, \quad \text{if } 2^j \big\lvert  \lvert\Arg(\nu_d + \mathrm{i}\nu_{d-1})\rvert -r\big\rvert \in I, \; j \geq j_0.
\end{equation}
The dependencies of the constants are $c_0= c_0(\kappa, \varepsilon, \tau)$, $c_1 = c_1(\kappa, \varepsilon, \tau)$, $j_0 = j_0(\kappa, \varepsilon, \tau)$, and $I = I(\kappa, \tau)$. In particular, none of the constants depend on $\nu$.
\end{theorem}
\begin{proof}
The following proof uses methods very similar to those in the one-dimensional setting \cite{bib14}. For $\nu_d^2 + \nu_{d-1}^2=1$, \hyperref[lemma testsignal coeff]{Lemma~\ref{lemma testsignal coeff}} and  \hyperref[lemma rep of rot curv]{Lemma~\ref{lemma rep of rot curv}} yield
\begin{align*}
 &\langle f_{r, \tau}, T^d(g_\nu^{-1})\Psi_{\scriptscriptstyle \mathrm{C}}^{j} \rangle_{\mathbb{S}^{d-1}} \\
 &\qquad \qquad  = \frac{ \gamma(\lambda, r, \tau)}{\sqrt{2^{-1}\Gamma(\frac{d}{2})}}  \sum_{n=0}^\infty \kappa\!\left(\frac{n}{2^{j-1}} \right) n^{(d-2)/4-\tau-1} \cos(n \Arg(\nu_d + \mathrm{i}\nu_{d-1})) \cos(nr-\varphi)  \\
 & \qquad \qquad  \quad + \mathcal{O}(2^{j((d-2)/4 - \tau - 1)}), \qquad j \rightarrow \infty.
\end{align*}
Also, by \hyperref[lemma rep of rot curv]{Lemma~\ref{lemma rep of rot curv}}, the value of $\langle f_{r, \tau}, T^d(g_\nu^{-1})\Psi_{\scriptscriptstyle \mathrm{C}}^{j} \rangle_{\mathbb{S}^{d-1}} $ does not depend on the sign of $\Arg(\nu_d+\mathrm{i}\nu_{d-1})$. Hence, for the remainder of the proof we can assume that $\Arg(\nu_d+\mathrm{i}\nu_{d-1}) \in [0, \pi]$. A straightforward calculation shows that
\begin{align*}
    &\frac{ \gamma(\lambda, r, \tau)}{\sqrt{2^{-1}\Gamma(\frac{d}{2})}}  \sum_{n=0}^\infty \kappa\!\left(\frac{n}{2^{j-1}} \right) n^{(d-2)/4-\tau-1} \cos(n \Arg(\nu_d + \mathrm{i}\nu_{d-1})) \cos(nr-\varphi) \\
    &\qquad \qquad = 2^{j((d-2)/4 - \tau)} \frac{2\pi\rho(\lambda, r, \tau)}{2^{j}} \sum_{n=0}^{2^j} \kappa_\tau \!\left(\frac{2\pi n}{2^j} \right) \cos(n(\Arg(\nu_d+\mathrm{i}\nu_{d-1}) -r) + \varphi) \\
    & \qquad \qquad \quad + 2^{j((d-2)/4 - \tau)} \frac{2\pi\rho(\lambda, r, \tau)}{2^{j}} \sum_{n=0}^{2^j} \kappa_\tau \!\left(\frac{2\pi n}{2^j} \right) \cos(n(\Arg(\nu_d+\mathrm{i}\nu_{d-1}) +r) - \varphi),
\end{align*}
where
\begin{equation*}
    \kappa_\tau(t) = \kappa(t/\pi) t^{\frac{d-2}{4}-\tau-1}.
\end{equation*}
Applying \cite[Lemma~5]{bib14}, we obtain
\begin{align*}
   &\frac{1}{2^{j}} \sum_{n=0}^{2^j} \kappa_\tau \!\left(\frac{2\pi n}{2^j} \right) \cos(n(\Arg(\nu_d+\mathrm{i}\nu_{d-1}) -r) + \varphi) \\
   &\qquad \qquad = \frac{1}{2\pi} \int_{0}^{2\pi} \kappa_\tau(t) \cos\!\left( \frac{2^j t (\Arg(\nu_d+\mathrm{i}\nu_{d-1})-r)}{2 \pi} + \varphi \right) \mathrm{d}t + \mathcal{O}(2^{-jq}).
\end{align*}
Similarly, \cite[Lemma~5]{bib14} combined with repeated integration by parts shows that
\begin{equation*}
    \frac{1}{2^{j}} \sum_{n=0}^{2^j} \kappa_\tau \!\left(\frac{2\pi n}{2^j} \right) \cos(n(\Arg(\nu_d+\mathrm{i}\nu_{d-1}) +r) - \varphi) = \mathcal{O}(2^{-jq}).
\end{equation*}
This completes the proof of \eqref{asympt formula main}.

To verify \eqref{lower bound main}, we note that the map
\begin{equation*}
    z \mapsto \int_{0}^{2\pi} \kappa(t/\pi)t^{(d-2)/4 - \tau-1} \cos\!\left(z t + \varphi\right)  \mathrm{d}t
\end{equation*}
is an entire function. Thus, a corresponding interval which is free of zeroes must exist.
\end{proof}

\autoref{thm detect 1} implies, in particular, that the curvelet coefficients exceed a certain threshold whenever the analysis function sufficiently matches the boundary $\partial C(e^d, r)$. We now establish a second result, showing that the values $\lvert \langle f_{r, \tau}, T^d(g_\nu^{-1})\Psi_{\scriptscriptstyle \mathrm{C}}^{j} \rangle_{\mathbb{S}^{d-1}}\rvert$ decay rapidly when the curvelet is either away from a singularity or not aligned with the orientation of $\partial C(e^d, r)$. Otherwise, if $T^d(g_\nu^{-1})\Psi_{\scriptscriptstyle \mathrm{C}}^{j}$ matches the boundary in terms of position and orientation, the upper bound given by the following theorem is consistent with the lower bound from \autoref{thm detect 1}. Again, the following estimate is motivated by a previous one-dimensional result from \cite[Theorem~2]{bib14}, and a two-dimensional version was previously established in \cite[Theorem~3.3]{bib26}.

\begin{theorem}\label{thm detect 2}
Let $\varepsilon>0$ and let $ \Psi_{\scriptscriptstyle \mathrm{C}}^{j}$, $j \in \mathbb{N}_0$, be as defined in \autoref{sec3} with $\kappa \in C^q([0, \infty))$, for some $q \in \mathbb{N}$. We assume that there is some $t_0 \in (1/2, 2)$ such that $\kappa^{(q)}(t) \neq 0$ for each $t \in (1/2, t_0)$. Then there exists a constant $c_2 = c_2(\kappa, \varepsilon, \tau, q)>0$ such that
\begin{equation*}
2^{j(\tau-(d-2)/4)}\lvert \langle f_{r, \tau}, T^d(g_\nu^{-1})\Psi_{\scriptscriptstyle \mathrm{C}}^{j} \rangle_{\mathbb{S}^{d-1}}\rvert \leq \frac{c_2 \lvert \nu_d + \mathrm{i}\nu_{d-1} \rvert^{2^{j-2}}}{(1+2^j \big\lvert r- \lvert\Arg(\nu_d + \mathrm{i}\nu_{d-1})\rvert \big\rvert)^q},
\end{equation*}
for all $\nu \in \mathbb{S}^{d-1}$, provided that $r \in [\varepsilon, \pi-\varepsilon]$.
\end{theorem}
\begin{proof}
Applying \hyperref[lemma testsignal coeff]{Lemma~\ref{lemma testsignal coeff}} and \hyperref[lemma rep of rot curv]{Lemma~\ref{lemma rep of rot curv}}, we obtain
\begin{align*}
    &2^{j(\tau-(d-2)/4)}\langle f_{r, \tau}, T^d(g_\nu^{-1})\Psi_{\scriptscriptstyle \mathrm{C}}^{j} \rangle_{\mathbb{S}^{d-1}} \\
    & \qquad = \Re\mleft\{\sum_{p=0}^{q-1} \frac{\phi_p(r)}{2^{j(1+p)}}\sum_{n=0}^\infty \kappa_{\tau, p} \mleft( \frac{ n}{2^{j-1}} \mright) \exp(\mathrm{i}nr) \cos(n \Arg(\nu_d + \mathrm{i}\nu_{d-1}) )  \mright\}+\mathrm{E}^j(\nu),
\end{align*}
where $\kappa_{\tau, p}(t) = \kappa(t)t^{(d-2)/4 -p-\tau-1} z^{2^{j-1}t}$, $z = \lvert \nu_d + \mathrm{i} \nu_{d-1} \rvert$, and
\begin{equation*}
    \lvert \mathrm{E}^j(\nu) \rvert \leq c  \lvert \nu_d + \mathrm{i}\nu_{d-1}\rvert^{2^{j-2}} 2^{-jq} .
\end{equation*}
The assertion now follows from the same arguments that we used to verify \eqref{key estimate} in the proof of \autoref{theorem curvelets localization}.
\end{proof}

Combining \autoref{thm detect 1} and \autoref{thm detect 2}, it follows that the (higher-order) singularities of $f_{r, \tau}$ can be identified precisely, in terms of position and orientation, by the asymptotic decay of the frame coefficients $\langle f_{r, \tau}, T^d(g)\Psi_{\scriptscriptstyle \mathrm{C}}^{j} \rangle_{\mathbb{S}^{d-1}}$. In particular,
\begin{equation*}
    \sup_{g \in SO(d)} \lvert \langle f_{r, \tau}, T^d(g)\Psi_{\scriptscriptstyle \mathrm{C}}^{j} \rangle_{\mathbb{S}^{d-1}} \rvert \sim 2^{-j(\tau - (d-2)/4)}.
\end{equation*}

\appendix
\addcontentsline{toc}{section}{Appendix}
\section*{Appendix}\label{appendix}

\setcounter{equation}{0}\renewcommand\theequation{A.\arabic{equation}}
\renewcommand\theHequation{A.\arabic{equation}} 


In the following, we briefly review some fundamentals of harmonic analysis on $SO(d)$. All formulas presented here can be found in classical literature such as \cite{bib32}. We also refer to our recent discussions in \cite{bib37, bib77}.

Let $\mu_d$ denote the Haar measure on $SO(d)$ with the normalization 
\begin{equation*}
\int_{SO(d)} \mathrm{d}\mu_d = 1.
\end{equation*}
Each element $g\in SO(d)$ can be written as $g = g_\eta h$, where $h \in SO(d-1)$ and $g_\eta \in SO(d)$ with $g_\eta e^{d} = \eta \in \mathbb{S}^{d-1}$. Clearly, this representation is not unique. However, independent of the choice of $g_\eta$ and $h$, it holds that
\begin{equation}\label{integral onion}
\int_{SO(d)} f(g) \, \mathrm{d}\mu_d(g) = \int_{\mathbb{S}^{d-1}} \int_{SO(d-1)} f(g_\eta h)  \, \mathrm{d}\mu_{d-1}(h) \, \mathrm{d}\omega_{d-1}(\eta),
\end{equation}
for each $f \in L^1(SO(d))$. Here, due to the invariance of the Haar measure $\mu_{d-1}$ on $SO(d-1)$, the inner integral on the right-hand side does not depend on the choice of $g_\eta$.

We now consider the unitary group representation
\begin{equation*}
g \mapsto T^d(g), \quad T^d(g)f(x) = f(g^{-1}x),
\end{equation*}
of $SO(d)$ on $L^2(\mathbb{S}^{d-1})$. Its subrepresentations on the spaces $\mathcal{H}_n^d$  are irreducible. Consequently, as part of the famous Peter-Weyl theorem, the matrix functions $t_{k, m}^{d,n}\colon SO(d) \rightarrow \mathbb{C}$ given by
\begin{equation}\label{matrix fct def}
t_{k, m}^{d,n}(g) = \langle T^d(g) Y_m^{d,n}, Y_k^{d,n} \rangle_{\mathbb{S}^{d-1}}, \qquad k, m \in \mathcal{I}_n^d, \; n \in \mathbb{N}_0,
\end{equation}
form an orthogonal system for $L^2(SO(d))$ w.r.t.\ the inner product
\begin{equation*}
\langle f_1, f_2 \rangle_{SO(d)} = \int_{ SO(d)} f_1  \overline{f_2} \, \mathrm{d}\mu_d.
\end{equation*}
Moreover,
\begin{equation*}
\int_{SO(d)} \lvert t_{k, m}^{d,n} \rvert^2 \, \mathrm{d}\mu_d = \frac{1}{\dim \mathcal{H}_n^{d}}.
\end{equation*}
It is not difficult to verify that the matrix functions in \eqref{matrix fct def} satisfy the symmetry
\begin{align}\label{symmetry matrix fct 1}
\overline{t_{k, \ell}^{d, n}(g)} = t_{k^-, \ell^-}^{d, n}(g),
\end{align}
where $(k_1, \dots, k_{d-3}, k_{d-2})^-$ is defined as $(k_1, \dots,k_{d-3}, -k_{d-2})$. Similarly, it is easy to see that
\begin{equation}\label{symmetry matrix fct 2}
    t_{k, \ell}^{d, n}(g^{-1}) = t_{\ell^-, k^-}^{d, n}(g).
\end{equation}
It is obvious that \eqref{matrix fct def} can also be written as
\begin{equation}\label{matrix fct def 2}
Y_m^{d, n}(g^{-1}\eta)=\sum_{k \in \mathcal{I}_n^d} t_{k, m}^{d, n}(g) \, Y_k^{d, n}(\eta), \quad \eta \in \mathbb{S}^{d-1}.
\end{equation}
The fact that $T^d$ is a group homomorphism, i.e., $T^d(g \tilde{g}) = T^d(g)T^d(\tilde{g})$, where each $T^d(g)$ is a unitary operator on $L^2(\mathbb{S}^{d-1})$, implies the addition formula
\begin{equation}\label{matrix fct addition formula}
t_{k, \ell}^{d, n}(g \tilde{g}) = \sum_{m \in \mathcal{I}_n^d} t_{k, m}^{d, n}(g) \, t_{m, \ell}^{d, n}(\tilde{g}), \quad g, \tilde{g} \in SO(d).
\end{equation}
Choosing $\eta = e^d$ in \eqref{matrix fct def 2}, a simple calculation yields
\begin{equation}\label{matrix fct and spherical harmonics}
\overline{Y_m^{d, n}(\nu)} = \sum_{k \in \mathcal{I}_n^d} \overline{Y_k^{d, n}(e^d)} \, t_{m, k}^{d, n}(g_\nu).
\end{equation}
In particular, spherical harmonics are linear combinations of certain matrix functions, as $\overline{t_{m, k}^{d, n}}$ is clearly a linear combination of the $t_{r, s}^{d, n}$, $r, s \in \mathcal{I}_n^d$. 

Matrix functions of different dimensions are intimately connected. Specifically, we have
\begin{equation}\label{matrix fct restrict}
t_{k, \ell}^{d, n}(h) = \delta_{k_1, \ell_1} \, t_{(k_2, \dots, k_{d-2}), (\ell_2, \dots, \ell_{d-2})}^{d-1, k_1}(h), \qquad h \in SO(d-1), \quad d \geq 4,
\end{equation}
as well as
\begin{equation}\label{matrix fct restrict d=3}
t_{k, \ell}^{3, n}(h(\gamma)) = \delta_{k, \ell}  \exp(\mathrm{i}k\gamma), \qquad \gamma \in [0, 2\pi),
\end{equation}
where $h(\gamma)\in \mathbb{R}^{3\times 3}$ is a positive rotation by $\gamma$ in the $(x_1, x_2)$-plane.

For $N \in \mathbb{N}_0$, let
\begin{equation}\label{class 1 matrix fct}
\mathcal{M}_N^1(SO(d)) = \vspan\{ t_{k, m}^{d, n}, \; k, m \in \mathcal{I}_n^d, \; 0 \leq n \leq N \}.
\end{equation}
Here, the upper index $1$ in the notation $\mathcal{M}_N^1(SO(d))$ refers to the fact that this space is spanned by class $1$ matrix functions (see \cite{bib32} for more details).



\sloppy 

\section*{Literature}       
\addcontentsline{toc}{section}{Literature}
\printbibliography[heading=none]

\end{document}